\documentclass[11pt]{amsart}
\usepackage[T1]{fontenc}
\usepackage{lmodern,amsmath,amssymb,amsthm,mathtools,microtype,needspace}
\usepackage[margin=1in]{geometry}
\usepackage[colorlinks=true,linkcolor=blue,citecolor=blue,urlcolor=blue]{hyperref}
\numberwithin{equation}{section}
\newtheorem{theorem}{Theorem}[section]
\newtheorem{proposition}[theorem]{Proposition}
\newtheorem{lemma}[theorem]{Lemma}

\theoremstyle{remark}\newtheorem{remark}[theorem]{Remark}
\newcommand{\R}{\mathbb R}\newcommand{\Z}{\mathbb Z}\newcommand{\HH}{\mathbb H}
\newcommand{\secg}{\operatorname{sec}}\newcommand{\Hess}{\operatorname{Hess}}
\newcommand{\PD}{\operatorname{PD}}
\newcommand{\sig}{\operatorname{sign}}\newcommand{\Ad}{\operatorname{Ad}}
\newcommand{\tr}{\operatorname{tr}}\newcommand{\Lie}{\mathcal L}

\title[Positive curvature and non-isometric circle actions]{Positive sectional curvature and non-isometric circle actions on a family of eleven-spheres}
\author[S. Deng]{Shaoqiang Deng}
\address[Shaoqiang Deng]{School of Mathematical Sciences, Nankai University, Tianjin, China}
\author[Z. Hu]{Zhiguang Hu}
\address[Zhiguang Hu]{School of Mathematical Sciences, Tianjin Normal University, Tianjin, China}
\author[H. Zhang]{Hui Zhang}
\address[Hui Zhang]{School of Mathematical Sciences, Southeast University, Nanjing, China}
\date{Research manuscript, October 5, 2026}
\begin{document}
\begin{abstract}
We construct Riemannian metrics of positive sectional curvature on a
family of homotopy eleven-spheres whose Eells--Kuiper invariants form
the even subgroup of $\Z/992$, and realize these metrics by
non-isometric circle actions. On each member there are a smooth
background metric $q$ and three effective circle actions, with
generators $W_1,W_2,W_3$, such that
$g^{-1}=q^{-1}+\sum_{a=1}^3W_a\otimes W_a$ defines a positively
curved metric. Each action is non-isometric for every partial metric,
including its incoming metric and the final metric. Gauge
transformations of the quaternionic Hopf bundle describe the sphere.
Compatible disk metrics and inverse-metric smoothing first produce
a positively curved comparison metric. A local conjugation then
produces a nearby positively curved metric with the required
non-isometric realization. For each fixed member we obtain the
explicit lower-bound formula
$2^{-54}(1+M_{0,k}+M_{1,k})^{-28}$ in terms of the curvature and
first covariant derivative norms of a chosen connection. These input
norms need not be the same for different choices on the same member;
the formula is not an evaluated numerical bound depending on $k$ alone.
\end{abstract}
\maketitle

\section{Introduction}\label{sec:introduction}
The existence of Riemannian metrics with positive sectional curvature is
a basic problem in differential geometry. For a manifold homeomorphic to
a sphere, the problem also involves its differentiable structure. A
homotopy sphere may carry an exotic smooth structure, and a metric on the
standard sphere cannot simply be transported to that structure. In a
construction by two disks, this distinction is recorded by the
diffeomorphism used to identify their boundaries.

Group actions provide useful ways to construct and deform metrics. When
an action is isometric, a quotient metric can be studied by the curvature
formula for a Riemannian submersion. For a non-isometric action, one first
has to specify how the action enters the metric. In this paper we use
the inverse-metric expression
\[
 g^{-1}=q^{-1}+W\otimes W.
\]
For every Riemannian metric $q$ and smooth vector field $W$, the right
side is positive definite on covectors. It therefore defines a new metric,
whether or not $W$ is Killing. If $W$ generates a circle action, this
is a metric deformation associated with that action.

There are two separate questions. The first is whether the resulting
metric has positive sectional curvature. The second is whether the
fields in the formula generate genuine periodic actions and are
non-Killing for the corresponding metrics. Positivity of an inverse metric
does not imply positivity of sectional curvature. Also, an arbitrary
change in the length of a periodic vector field may destroy periodicity.
Both issues must be addressed in the construction.

Our construction applies to a family of oriented homotopy eleven-spheres.
Their smooth types are distinguished by the Eells--Kuiper invariant,
with the normalization in \eqref{eq:mu}. The main result is as follows.

\Needspace{17\baselineskip}
\begin{theorem}\label{thm:main}
For each $k\in12\Z$, let $\Sigma_k^{11}$ be the oriented homotopy
eleven-sphere obtained by the bundle and fiber-surgery construction in
Section~\ref{sec:topology}. There exist a smooth background metric
$q_k$ and three effective smooth circle actions of period $2\pi$, with
generators $W_{1,k},W_{2,k},W_{3,k}$, such that
\begin{equation}\label{eq:main}
 g_k^{-1}=q_k^{-1}+\sum_{a=1}^3 W_{a,k}\otimes W_{a,k}
\end{equation}
defines a metric of positive sectional curvature. For every subset $S\subset\{1,2,3\}$ and every $a=1,2,3$, define
\[
 (g_{S,k})^{-1}=q_k^{-1}+\sum_{b\in S}W_{b,k}\otimes W_{b,k}.
\]
Then $\Lie_{W_{a,k}}g_{S,k}\ne0$. In particular, each action is
non-isometric for its incoming metric, in any order of addition, and
for the final metric.

The parameters can be chosen so that, at every point and on every
two-plane,
\begin{equation}\label{eq:mainexplicit}
 \secg_{g_k}>\frac1{2^{54}(1+M_{0,k}+M_{1,k})^{28}}>0.
\end{equation}
Here $M_{0,k}=\sup|\Omega_k|$ and $M_{1,k}=\sup|D\Omega_k|$ refer
to the chosen southern connection used for $\Sigma_k$ in
Section~\ref{sec:disks}, with its base metric normalized to curvature
one. They are norms of those input data, not invariants of the smooth
type or numerical quantities specified by $k$ alone. As $k=12m$
varies, $\mu(\Sigma_k)=2m/992$ runs through all $496$ elements of the
even subgroup of $\Theta_{11}\cong\Z/992$. The background is defined
from the connection and the disk profiles before the final circle
actions are chosen. The lower bound is positive for every fixed $k$;
this theorem does not assert a single uniform bound for all $k$.
\end{theorem}

For each fixed $k$, the three actions are subgroups of a common
conjugate of an $S^3$ action. They need not commute. The non-isometry assertion holds for every partial sum, including the
initial background and the final metric. The constants $M_{0,k},M_{1,k}$ are finite for every
fixed smooth connection, so \eqref{eq:mainexplicit} is a strictly
positive lower bound on all points and two-planes of that particular
metric, although its value has not been evaluated numerically.

We now explain the proof. First, for each $k\in12\Z$, we choose a
principal $S^3$ bundle over $S^4\times S^4$ with second Chern class
$x+ky$. Its transition
functions give a family of gauge transformations of the Hopf bundle
$S^7\to S^4$. Surgery on an $S^7$ fiber gives a homotopy eleven-sphere.
The gauge family describes its attaching diffeomorphism, and the
characteristic numbers of the surgery filling identify its smooth type.

Second, we construct compatible metrics on the two disks. The southern
metric comes from a connection metric with a small variable fiber
radius. A concavity condition supplies positive mixed curvature.
On the northern side we use a flat connection and a different radial
profile. We estimate the curvature on every plane horizontal for the
quotient map. The boundary metrics agree under the prescribed attaching
map, and the outward second fundamental forms have positive definite sum.

\paragraph{Relation to He--Liu--Yau.}
The two-disk architecture in this paper is closely related to the one
used by He, Liu and Yau \cite{HLY} to construct positively curved
metrics on homotopy seven-spheres. In particular, we use the same broad
geometric pattern: a connection metric with a variable fibre radius on
one disk, a flat-connection radial model on the other, compatible
boundary metrics with a favorable sum of outward second fundamental
forms, and a positive-curvature gluing theorem. We do not present this
two-disk architecture as new. Our work is to carry it through for the
quaternionic bundle and fibre-surgery data defining the eleven-dimensional
family: we establish dimension-independent curvature estimates,
identify the resulting smooth structures, and verify the boundary data
for the prescribed attaching maps. We then add a further metric
representation, $g^{-1}=q^{-1}+\sum_a W_a\otimes W_a$, which survives
our explicit smoothing, and construct periodic circle actions that are
non-isometric for every partial metric. The gluing theorem used here is
that of Reiser and Wraith \cite{RW}; the commuting-action description
is also related to Speran\c ca \cite{Speranca}. Thus the present paper
uses and develops the He--Liu--Yau two-disk framework in a different
dimensional and topological setting, while adding the action-preserving
smoothing and non-isometric realization results.

The inverse metric is useful for this last purpose. We enlarge the
fiber radii to define a background on each disk. Adding three circle
squares to its inverse recovers the positive comparison metric.
The sum of those squares is independent of the normal coordinate in a
common collar. It therefore commutes with convolution of the inverse
metric. This gives a global positive comparison metric and preserves
the exact formula. Reiser--Wraith's gluing theorem already preserves
arbitrary strict curvature lower bounds \cite[Corollary B]{RW}.
The additional requirement here is compatibility of a prescribed
inverse-metric smoothing with this fixed tensor sum, while retaining
the periodic symmetries of the background.

Finally, a Gram-matrix shear conjugates the three actions by the flow
of a vector field supported in a small coordinate ball. Conjugation
preserves their periods and group laws. A second-derivative calculation
gives an exact nonzero Lie derivative for every partial metric.
Smooth dependence and compactness then retain a positive curvature
bound on the whole closed manifold; Appendix~\ref{app:coordinatecurvature}
records the corresponding coordinate formulas.

The two-disk comparison metric supplies the positive curvature margin.
Keeping the background fixed, the conjugation produces a nearby
positively curved metric realized by periodic circle actions whose
generators are non-Killing for every partial sum. The comparison
metric and the deformed metric are generally different:
$\gamma_\zeta^{-1}=q_\zeta^{-1}+T$, whereas
$g_\tau^{-1}=q_\zeta^{-1}+T_\tau$.
The smallness of $\tau$ preserves the positive curvature margin;
non-isometry is a property of the resulting realization, rather than
the source of that margin. The background is part of our construction;
it is not asserted to be the Euclidean-induced metric on an algebraic
model of the sphere.

The paper is arranged as follows. Section~\ref{sec:preliminaries}
contains the metric identities and curvature conventions.
Section~\ref{sec:topology} constructs and identifies the sphere.
The connection estimate is proved in Section~\ref{sec:connection},
and the compatible disk metrics are constructed in
Section~\ref{sec:disks}. In Section~\ref{sec:background}, we define
the background and smooth its inverse. The non-isometric actions and
the main theorem are treated in Section~\ref{sec:circles}.
Section~\ref{sec:quantitative} gives definite parameters and the
lower bound \eqref{eq:mainexplicit}. We finish by identifying the
Brieskorn model and discussing the scope of the construction.
Appendix~\ref{app:coordinatecurvature} contains the local curvature
calculation for the conjugated metrics.

\section{Preliminaries}\label{sec:preliminaries}
We first explain the deformation by circle squares and fix the conventions
used in the curvature calculations. All geometric data are smooth unless
otherwise stated.

\subsection{Curvature and tensor norms}
Our curvature convention is
\[
 R(X,Y)Z=\nabla_X\nabla_YZ-\nabla_Y\nabla_XZ-\nabla_{[X,Y]}Z.
\]
For independent tangent vectors $X,Y$, put
\[
 \mathcal K(X\wedge Y)=g(R(X,Y)Y,X),\qquad
 \secg(X,Y)=\frac{\mathcal K(X\wedge Y)}{|X\wedge Y|_g^2}.
\]
Thus positive sectional curvature means positivity of $\mathcal K$ on
nonzero decomposable bivectors. It does not require the curvature
operator to be positive on every bivector.

We identify $S^3$ with the unit quaternions. Its Lie algebra
$\operatorname{Im}\HH$ has the Euclidean inner product $Q$, so the
bi-invariant group metric has curvature one. For an oriented orthonormal
basis, $[e_a,e_b]=2\epsilon_{abc}e_c$. Each subgroup
$s\mapsto\exp(se_a)$ has period $2\pi$. The notation $h_{S^j}$
denotes the unit round metric on $S^j$.

Tensor norms are induced Hilbert--Schmidt norms, except where an
operator norm is specified. For an adjoint-bundle-valued two-form,
\[
 |\Omega|^2=\sum_{i<j,a}(\Omega_{ij}^a)^2,\qquad
 |D\Omega|^2=\sum_{k,i<j,a}((D_k\Omega)_{ij}^a)^2.
\]
Hence ordered component arrays have norms $\sqrt2|\Omega|$ and
$\sqrt2|D\Omega|$. Exterior powers carry their induced norms.
For a fixed metric $h$, the $C^2(h)$ norm of a tensor is the maximum
of the suprema of its norm and those of its first two covariant
derivatives.

\subsection{Adding a circle square to an inverse metric}
The inverse of a metric $q$ is a positive bilinear form on covectors.
For a vector field $W$, the tensor $W\otimes W$ is defined by
\[
 (W\otimes W)(\alpha,\beta)=\alpha(W)\beta(W).
\]
The following elementary identity gives a covariant expression for the
deformation.

\begin{lemma}\label{lem:rankone}
For every Riemannian metric $q$ and smooth vector field $W$, the
formula $g^{-1}=q^{-1}+W\otimes W$ defines a smooth Riemannian metric,
and
\begin{equation}\label{eq:rankone}
 g=q-\frac{q(W,\cdot)\otimes q(W,\cdot)}{1+q(W,W)}.
\end{equation}
\end{lemma}
\begin{proof}
For a nonzero covector $\alpha$,
\[
 (q^{-1}+W\otimes W)(\alpha,\alpha)
 =q^{-1}(\alpha,\alpha)+\alpha(W)^2>0.
\]
At a fixed point choose a $q$-orthonormal basis, and let $w$ be the
column vector of $W$. Then
\[
 (I+ww^T)\left(I-\frac{ww^T}{1+w^Tw}\right)=I.
\]
This proves \eqref{eq:rankone}. Its denominator is positive, so it is
smooth even at the zeros of $W$.
\end{proof}

For several vector fields we apply this operation successively. Their
squares add in the inverse metric. The final metric is independent of
the order of addition, but the non-Killing condition at an incoming
stage depends on that order. We shall verify it for every partial sum, and hence for every
ordering of the three additions.

An action generated by $W$ is isometric for $q$ precisely when
$\Lie_Wq=0$. We use the term \emph{realization by circle actions}
for an inverse-metric formula with genuinely periodic generators.
This definition makes sense for non-isometric actions. At this stage
we are not taking a Riemannian quotient of a non-invariant product metric.

\subsection{Submersions and the boundary convention}
For a principal $S^3$ bundle $\pi:P\to B$, a connection $\omega$,
a base metric $h$ and a positive function $r$, write
\[
 G=\pi^*h+r^2Q(\omega,\omega).
\]
The connection-horizontal space and the principal-vertical space are
orthogonal. The principal fibers have round metrics of radius $r$.

Our source bundles also carry a second free $S^3$ action, commuting
with the principal one. If it preserves $G$, its quotient has a
Riemannian metric. The quotient-horizontal spaces are orthogonal to the
orbits of this second action; they generally differ from the
connection-horizontal spaces. This distinction is essential for the
northern disk.

For orthonormal quotient-horizontal vectors $\widetilde X,\widetilde Y$,
O'Neill's formula \cite{ONeill} gives
\[
 \secg_{\mathrm{quotient}}(d\varpi\widetilde X,d\varpi\widetilde Y)
 =\secg_G(\widetilde X,\widetilde Y)
      +3|\mathcal A_{\widetilde X}\widetilde Y|^2
 \ge\secg_G(\widetilde X,\widetilde Y).
\]
It is therefore sufficient to prove a lower bound on all lifted
two-planes.

For a boundary with outward unit normal $n$, our convention is
$\mathrm{II}(X,Y)=g(\nabla_Xn,Y)$. In a normal coordinate with metric
$dt^2+h(t)$ and outward normal $\partial_t$, this is $h'(t)/2$.
For gluing we compare the sum of the outward second fundamental forms
under the prescribed boundary identification.

\section{The eleven-sphere family}\label{sec:topology}
We need a two-disk description and an identification of its smooth
type. A family of gauge transformations supplies the boundary map.
We then compute the characteristic numbers of a filling associated
with the same surgery.
\subsection{Principal bundles over a product}
Orient $B=S^4_x\times S^4_y$ and choose generators $x,y\in H^4(B;\Z)$ with
\[
 x^2=y^2=0,\qquad \langle xy,[B]\rangle=1.
\]
For a quaternionic line bundle, use its complex orientation as a complex
rank-two bundle. Its real characteristic classes then satisfy
\begin{equation}\label{eq:lineclasses}
 e=c_2,\qquad p_1=-2c_2,\qquad p_2=c_2^2.
\end{equation}

\begin{lemma}\label{lem:bundle}
For every $k\in12\Z$, there is a smooth principal $\mathrm{SU}(2)$ bundle
$E_k\to B$ whose associated quaternionic line bundle $\xi_k$ has
$c_2(\xi_k)=x+ky$.
\end{lemma}
\begin{proof}
The CW structure of $B$ is obtained from $S^4\vee S^4$ by attaching an
$8$-cell along the Whitehead product of the two inclusions. Let
$\iota\in\pi_4(B\mathrm{SU}(2))\cong\Z$ classify the Hopf bundle. Specify
$\iota$ and $k\iota$ on the two wedge summands. The obstruction to extension
is $[\iota,k\iota]=k[\iota,\iota]$ in
\[
 \pi_7(B\mathrm{SU}(2))=\pi_6(S^3)\cong\Z/12.
\]
The last group is the classical homotopy-group calculation \cite{Toda}.
The obstruction vanishes when $12\mid k$. There are no cells of higher
dimension. The resulting bundle can be made smooth; its second Chern class
is determined by the two restrictions. No uniqueness of the extension is
needed.
\end{proof}

Regard the total space $E_k=S(\xi_k)$ as a bundle over $S^4_y$. Its fiber is
the total space of the Hopf bundle over $S^4_x$, hence is the standard $S^7$.
Choose the convention in which the principal $S^3$ action on this $S^7$ is
left scalar multiplication on $\HH^2$.

Over each closed hemisphere of $S^4_y$, the principal bundle over
$S^4_x\times D^4$ is isomorphic to the pullback of the fixed Hopf bundle.
Such isomorphisms may, for example, be obtained by invariant parallel
transport along radial paths in $D^4$. The transition is therefore a smooth
family
\begin{equation}\label{eq:gaugefamily}
 f:S^3\longrightarrow\mathcal G,\qquad
 \mathcal G=\operatorname{Gauge}(S^7\longrightarrow S^4).
\end{equation}
Each $f_p$ is an orientation-preserving diffeomorphism of $S^7$ covering the
identity of $S^4$ and commuting with left scalar multiplication.
Changing one trivialization by a constant gauge transformation makes $f$
based at the identity. A smooth homotopy then makes it the identity on a
neighborhood of its parameter basepoint.

\subsection{Fiber surgery}
Perform surgery on one $S^7$ fiber of $E_k\to S^4_y$, using the product
normal framing given by a base disk. Let $W=D(\xi_k)$ with its orientation
induced by $B$ and the quaternionic line, and let $Z$ be obtained by
attaching the corresponding index-eight handle. The resulting boundary
has the presentation
\begin{equation}\label{eq:fibersurgery}
 \Sigma_k=(S^7\times D^4)\ \cup_{\Phi_f}\ (D^8\times S^3),
 \qquad \Phi_f(w,p)=(f_p(w),p),
\end{equation}
where $\Phi_f$ is the surgery attaching map. Orient the handle so that
it extends the orientation of $W$ across the oriented surgery, and orient
$\partial Z$ by the outward-normal-first convention. We use this
orientation on $\Sigma_k=\partial Z$ throughout; the oriented
comparison with the two-disk model is made explicitly below.

\begin{lemma}\label{lem:sphere}
The manifold in \eqref{eq:fibersurgery} is a homotopy $11$-sphere.
\end{lemma}
\begin{proof}
The pieces and their intersection $S^7\times S^3$ are simply connected.
In the Mayer--Vietoris sequence, the map from $H_7(S^7\times S^3)$ to
$H_7(S^7\times D^4)$ is an isomorphism because every $f_p$ has degree one.
The map from $H_3(S^7\times S^3)$ to $H_3(D^8\times S^3)$ is an isomorphism
because $\Phi_f$ preserves $p$. All other intermediate groups vanish, and
$\Sigma_k$ has the integral homology of $S^{11}$. By the Hurewicz theorem applied successively, its homotopy groups
below dimension eleven vanish. A generator in dimension eleven is
represented by a map from $S^{11}$ which is a homology equivalence.
The homological form of the Whitehead theorem makes it a homotopy
equivalence, since both spaces are simply connected.
\end{proof}

\subsection{The equivariant two-disk presentation}
Write the parameter sphere as $D^3/\partial D^3$ and choose the gauge family
to be the identity on a collar of the boundary of this parameter disk.
Principal division gives a unique smooth map $b(p,w)\in S^3$ such that
\begin{equation}\label{eq:division}
 f_p(w)=b(p,w)w,\qquad b(p,qw)=q b(p,w)q^{-1}.
\end{equation}
For $w=(w_1,w_2)\in S^7$, principal division has the concrete formula
$b(p,w)=f_{p,1}(w)\overline w_1+f_{p,2}(w)\overline w_2$.
The second identity follows from $f_p(qw)=qf_p(w)$.

Let
\[
 V=\R^3\oplus\HH^2,\qquad \rho(q)(p,w)=(p,qw),\qquad S(V)=S^{10}.
\]
Choose $0<r_0<1$. On the subset $|p|\le r_0$ of $S(V)$ define
\[
 \beta(p,w)=b(p/r_0,w/|w|).
\]
Extend it by $1$ on $|p|\ge r_0$. Since $b$ equals $1$ on a parameter
boundary collar, this is smooth, including near $w=0$. It satisfies
\begin{equation}\label{eq:betaequiv}
 \beta(\rho(q)z)=q\beta(z)q^{-1}.
\end{equation}
Define
\begin{equation}\label{eq:J}
 J_\beta(z)=\rho(\beta(z))z.
\end{equation}
This is a diffeomorphism: its inverse is $J_{\beta^{-1}}$. Indeed,
$\beta(\rho(\beta(z))z)=\beta(z)$ by equivariance.

\begin{lemma}\label{lem:twodisk}
With the orientations on the disk pieces induced from the surgery boundary,
there is an orientation-preserving diffeomorphism
$\Sigma_k\cong D^{11}_{N}\cup_{J_\beta}D^{11}_{S}$.
Here the north-to-south boundary transition is $J_\beta$; reversing the
labels gives $J_\beta^{-1}$ and the same oriented sphere.
\end{lemma}
\begin{proof}
The standard decomposition
\[
 S^{10}=(S^7\times D^3)\cup(S^2\times D^8)
\]
identifies $J_\beta$ with $(w,p)\mapsto(f_p(w),p)$ on its first part and
with the identity on its second part. The radial rescaling in the definition
of $\beta$ changes only the parameterization of these pieces.

Express $D^4$ as a rounded $D^3\times[-1,1]$, with the support of $f$
on the top face of its boundary. Cut the oriented surgery boundary
$\partial Z=\Sigma_k$ along the middle copy
\[
 (S^7\times D^3\times\{0\})\cup(D^8\times S^2\times\{0\}).
\]
Call the upper and lower halves $Y_N$ and $Y_S$, and give them the
orientations restricted from $\partial Z$. The lower half is a standard
disk, with identity straightening coordinate $\sigma_S$. On the
cylindrical part of the upper half, the straightening coordinate from the
surgery coordinates to disk coordinates is
\[
 \sigma_N(w_{\rm old},p,s)=(f_p^{-1}(w_{\rm old}),p,s),
\]
where $s$ is the interval coordinate. It is the identity near the side
boundary and cancels the top-face attaching map, so it extends over the
upper half as a diffeomorphism to a standard disk. Every $f_p$ preserves
the orientation of $S^7$, so $\sigma_N$ and its inverse preserve the
product orientation on this cylindrical part. Let
$\chi_N=\sigma_N^{-1}:D_N^{11}\to Y_N$ and
$\chi_S=\sigma_S^{-1}:D_S^{11}\to Y_S$ be the disk identifications.
Orient the abstract disks by pullback of the orientations of $Y_N,Y_S$
under $\chi_N,\chi_S$. Thus both disk identifications preserve
orientation, and the boundary transition is orientation reversing for
the induced outward-normal-first boundary orientations, as required for
oriented gluing.

On the cutting face, if $w_N$ and $w_S$ denote the upper and lower disk
coordinates, then $w_N=f_p^{-1}(w_{\rm old})$ and $w_S=w_{\rm old}$;
hence $w_S=f_p(w_N)$. On $S^7\times D^3$ this is precisely $J_\beta$,
and on $D^8\times S^2$ it is the identity. The identification of the
cutting face is recorded by the commutative diagram
\[
\begin{array}{ccc}
 \partial D_N^{11} & \xrightarrow{\ J_\beta\ } & \partial D_S^{11}\\
 {\scriptstyle\chi_N}\downarrow && \downarrow{\scriptstyle\chi_S}\\
 Y & \xrightarrow{\ \mathrm{id}\ } & Y.
\end{array}
\]
Here $Y$ is the cutting hypersurface and
$\chi_N=\chi_S\circ J_\beta$ on the boundary. Together with the
outward-normal-first orientation of $\partial Z$, this specifies the sign
of the clutching map relative to the oriented filling, not just its
unoriented diffeomorphism type. The argument uses a family of
diffeomorphisms, not an extension of an individual gauge map to $D^8$.
\end{proof}

\subsection{The Eells--Kuiper invariant}
We now identify the sphere. The calculation is made on the filling
associated with the actual fiber surgery.

For an oriented homotopy eleven-sphere bounding a spin manifold $Z^{12}$,
normalize the Eells--Kuiper invariant by
\begin{equation}\label{eq:mu}
 \mu(\partial Z)=\frac{-3\langle\bar p_1p_1^2,[Z,\partial Z]\rangle
 +4\langle\bar p_1p_2,[Z,\partial Z]\rangle-24\sig(Z)}{190464}
 \pmod{\Z}.
\end{equation}
Here $\bar p_1$ is the rational relative lift. On
$\Theta_{11}=bP_{12}\cong\Z/992$ this invariant is injective
\cite{EK,KM,BGKT}; a parallelizable filling of signature eight has
invariant $-1/992$ in this convention.

The filling $Z$ is obtained from $W$ by the index-eight handle attachment
along the framed $S^7$ fiber described above, and $\partial Z=\Sigma_k$.
The attachment has index eight, so
\begin{equation}\label{eq:cohomhandle}
 H^4(Z;\Z)\longrightarrow H^4(W;\Z)\ \text{is an isomorphism},
 \qquad H^6(Z;\Z)=0.
\end{equation}
Also $H^2(Z;\Z/2)=0$, hence $Z$ is spin and $\sig(Z)=0$.

The zero section $B$ stays in the interior of $Z$, with normal bundle
$\xi_k$. Let $u=\PD_Z[B]\in H^4(Z,\partial Z;\Z)$. Its image in the absolute
cohomology of $W$ is $e=x+ky$. Since $TB$ is stably trivial,
\[
 p_1(TW)=-2e.
\]
Equation \eqref{eq:cohomhandle} implies that $p_1(TZ)=-2j(u)$, where $j$
is the relative-to-absolute map. As $\partial Z$ is a homotopy sphere, $j$
is an isomorphism in degree four. The relative lift is therefore exactly
\begin{equation}\label{eq:relativep1}
 \bar p_1=-2u.
\end{equation}
On $B$, the splitting $TZ|_B=TB\oplus\xi_k$ and
\eqref{eq:lineclasses} give $p_1=-2e$ and $p_2=e^2$. Consequently
\begin{align}
 \langle\bar p_1p_1^2,[Z,\partial Z]\rangle
 &=-2\langle4e^2,[B]\rangle=-16k,\label{eq:p13}\\
 \langle\bar p_1p_2,[Z,\partial Z]\rangle
 &=-2\langle e^2,[B]\rangle=-4k.\label{eq:p1p2}
\end{align}
These calculations do not require knowing $p_2$ on the additional
$8$-cycle created by the handle: the relative $p_1$ is supported on the
zero section in the Poincare-duality sense.
Substitution in \eqref{eq:mu} yields
\begin{equation}\label{eq:computedmu}
 \mu(\Sigma_k)=\frac{48k-16k}{190464}=\frac{k}{5952}.
\end{equation}
For $k=12m$, this is $2m/992$. It runs through all $496$ elements of the
even subgroup. Injectivity of $\mu$ shows that two parameters give
the same oriented smooth type if and only if
$k-k'\in5952\Z$. Thus $k=12m$, $0\le m\le495$, gives one
representative of each type. In particular $k=-384$ gives
$-64/992=-2/31$, a class of order $31$.

\begin{remark}
The surgery here is on the $S^7$ fiber, not on an $S^4$ section of
$E_k\to S^4$. The latter is a different surgery and does not, by itself,
identify the clutching sphere. Equations \eqref{eq:relativep1}--\eqref{eq:p1p2}
are computed directly on the filling for the fiber surgery.
\end{remark}

For every $k\in12\Z$, the linear action $\rho$ on the two disks gives a
global $S^3$ action, because $J_\beta$ is equivariant. It is free
wherever the quaternionic coordinate is nonzero. A point of this kind
will be used in Section~\ref{sec:circles}. The case $k=-384$ is the
particular example with $\mu=-2/31$.

\section{Curvature of a connection metric}\label{sec:connection}
We establish an estimate valid in every base dimension. The main point
is to control arbitrary two-planes, including the terms coupling
horizontal and vertical components.

\subsection{The curvature components}
Let $(B_0,g_0)$ be a compact base, possibly with boundary, with
$\secg_{g_0}\ge\kappa>0$, and let
$\omega$ be any smooth principal $S^3$ connection. Normalize the Lie algebra
inner product $Q$ so that the group has sectional curvature one, and put
\[
 G=\pi^*g_0+r^2Q(\omega,\omega),\quad
 \theta=d\log r,\quad N=-r^{-1}\Hess r.
\]
Let $D\Omega$ denote the covariant derivative of the curvature form,
using the base Levi--Civita connection and the induced adjoint-bundle
connection. Use
\[
 M_0=\sup|\Omega|,\qquad M_1=\sup|D\Omega|,
 \qquad |\Omega|^2=\sum_{i<j,a}(\Omega^a_{ij})^2,\quad
 |D\Omega|^2=\sum_{l,i<j,a}((D_l\Omega)^a_{ij})^2.
\]
For a decomposable bivector in normalized horizontal and vertical frames,
write
\[
 (X+U)\wedge(Y+V)=\alpha+\gamma+\nu,
 \quad \alpha=X\wedge Y,\quad
 \gamma=X\otimes V-Y\otimes U,\quad\nu=U\wedge V,
\]
and set $h=|\alpha|$, $m=|\gamma|$, $v=|\nu|$.

Whenever a quadratic lower bound in the variables $(h,m,v)$, or in a
rescaling of them, is represented by a symmetric coefficient matrix $Q$,
we use the standard convention
\[
 \begin{pmatrix}h&m&v\end{pmatrix}Q
 \begin{pmatrix}h\\m\\v\end{pmatrix}
 =Q_{11}h^2+Q_{22}m^2+Q_{33}v^2
   +2Q_{12}hm+2Q_{13}hv+2Q_{23}mv.
\]
Here $Q$ is the coefficient matrix of a scalar quadratic lower bound
in the nonnegative variables $h,m,v$; it is not the curvature operator
on the space of bivectors. Each pair of symmetric entries contributes
$Q_{ij}x_ix_j+Q_{ji}x_jx_i=2Q_{ij}x_ix_j$. Thus the entries
$Q_{12},Q_{13},Q_{23}$ are one half of the corresponding cross-term
coefficients. This convention also applies after the rescaling in
Lemma~\ref{lem:shrinking} and in Section~\ref{sec:quantitative}.

\begin{lemma}\label{lem:connectionbound}
The sectional numerator satisfies
\begin{align}\label{eq:curvaturebound}
 \mathcal K_G\ \ge{}&
 (\kappa-\tfrac34r^2M_0^2)h^2
 +(\lambda_{\min}N-\tfrac12r^2M_0^2)m^2
 +(r^{-2}-|\theta|^2)v^2\notag\\
 &-2r(M_1+2|\theta|M_0)hm
 -(3M_0+r^2M_0^2)hv-2r|\theta|M_0mv.
\end{align}
\end{lemma}
\begin{proof}
Let $X_i$ be horizontal lifts of a base orthonormal frame, and let
$E_a=r^{-1}e_a^\#$ for a fixed orthonormal basis of $\operatorname{Im}\HH$.
Use a base-normal frame and a local connection gauge normal at the point.
With $R(X,Y)=\nabla_X\nabla_Y-\nabla_Y\nabla_X-\nabla_{[X,Y]}$,
the relevant frame brackets are
\[
 [X_i,X_j]=-r\Omega_{ij}^aE_a,\qquad
 [X_i,E_a]=-\theta_iE_a,\qquad
 [E_a,E_b]=r^{-1}c_{ab}^cE_c
\]
at the chosen base-normal point. Substitution in the Koszul formula
gives the connection identities
\begin{align*}
 \nabla_{X_i}X_j&=(\nabla^0_iX_j)^H-\tfrac r2\Omega^a_{ij}E_a,&
 \nabla_{X_i}E_a&=\tfrac r2\Omega^a_{ij}X_j,\\
 \nabla_{E_a}X_i&=\tfrac r2\Omega^a_{ij}X_j+\theta_iE_a,&
 \nabla_{E_a}E_b&=\tfrac1{2r}c^c_{ab}E_c-\delta_{ab}\theta_iX_i.
\end{align*}
Here $[e_a,e_b]=c^c_{ab}e_c$, and adjoint equivariance gives
$E_a(\Omega^b_{ij})=-r^{-1}c^b_{ac}\Omega^c_{ij}$.
In differentiating these expressions, the vertical derivative of
$\Omega$ must be retained. Its adjoint equivariance is responsible
for the Lie-bracket terms in the first two lines below. The Hessian
term is $N_{ij}=-\nabla^0_i\theta_j-\theta_i\theta_j$.
Computing $R(A,B)C$ now gives the mixed components
\begin{align*}
 \langle R(X_i,E_a)E_b,X_j\rangle
 &=N_{ij}\delta_{ab}+\tfrac{r^2}4\sum_l\Omega^b_{il}\Omega^a_{jl}
   -\tfrac14\langle[e_a,e_b],\Omega_{ij}\rangle,\\
 \langle R(X_i,X_j)E_a,E_b\rangle
 &=\tfrac{r^2}4\sum_l(\Omega^a_{il}\Omega^b_{jl}-\Omega^a_{jl}\Omega^b_{il})
   +\tfrac12\langle[e_a,e_b],\Omega_{ij}\rangle,\\
 \langle R(X_i,E_a)E_b,E_c\rangle
 &=\tfrac r2\sum_l\theta_l(\delta_{ab}\Omega^c_{il}-\delta_{ac}\Omega^b_{il}),\\
 \langle R(X_i,X_j)X_l,E_a\rangle
 &=-\tfrac r2((D_i\Omega)^a_{jl}-(D_j\Omega)^a_{il})\\
 &\hspace{5mm}-\tfrac r2(\theta_i\Omega^a_{jl}-\theta_j\Omega^a_{il})
   +r\theta_l\Omega^a_{ij}.
\end{align*}
The horizontal sectional term is the base term minus
$3r^2|\Omega(X,Y)|^2/4$. The vertical curvature coefficient is
$r^{-2}-|\theta|^2$.

There is one essential algebraic point in contracting these components.
Decomposability gives
\[
 \gamma_{ia}\gamma_{jb}-\gamma_{ib}\gamma_{ja}
   =\alpha_{ij}\nu_{ab}.
\]
Thus the terms linear in $\Omega$ from the two mixed blocks combine to
$-\frac32\langle\Omega(X,Y),[U,V]\rangle$. Their absolute value is at most
$3M_0hv$, since $|[U,V]|=2|U\wedge V|$. They are not an uncontrolled
multiple of $m^2$.

We give the remaining contractions explicitly. Write $d=\dim B_0$.
Horizontal indices range from $1$ to $d$, and vertical indices from
$1$ to $3$. For the bilinear curvature form use
\[
 \mathcal R(A\wedge B,C\wedge D)=\langle R(A,B)D,C\rangle.
\]
Thus the off-diagonal block contributions are
$2\mathcal R(\alpha,\gamma)$, $2\mathcal R(\alpha,\nu)$ and
$2\mathcal R(\gamma,\nu)$.
Write $\Gamma=(\gamma_{ia})$ and $B_l=(\Omega^a_{il})_{i,a}$,
both $d\times3$ matrices. Denote the Hilbert--Schmidt matrix norm
by $\|\cdot\|_{\rm HS}$. Then
\[
 \|\Gamma\|_{\rm HS}=m,\qquad
 \sum_{l=1}^d\|B_l\|_{\rm HS}^2
 =\sum_{i,l,a}(\Omega^a_{il})^2=2|\Omega|^2.
\]

\emph{The quadratic mixed block.}
Its contribution after removing the $N$ and Lie-bracket terms is
\[
 S_m=\frac{r^2}{4}\sum_{i,j,a,b,l}
       \gamma_{ia}\gamma_{jb}\Omega^b_{il}\Omega^a_{jl}
     =\frac{r^2}{4}\sum_l\tr\bigl((\Gamma^TB_l)^2\bigr).
\]
For any real square matrix $C$, Cauchy--Schwarz gives
$|\tr(C^2)|\le\|C\|_{\rm HS}^2$. Therefore
\[
 |S_m|\le\frac{r^2}{4}\sum_l\|\Gamma^TB_l\|_{\rm HS}^2
 \le\frac{r^2}{4}m^2\sum_l\|B_l\|_{\rm HS}^2
 \le\frac{r^2}{2}M_0^2m^2.
\]

\emph{The quadratic horizontal--vertical block.}
Regard $B_l$ as a linear map from the vertical space to the horizontal
space. The quadratic part of $2\mathcal R(\alpha,\nu)$ is
\begin{align*}
 S_{hv}&=-\frac{r^2}{2}
   \sum_{i<j,\,a<b,\,l}\alpha_{ij}\nu_{ab}
    (\Omega^a_{il}\Omega^b_{jl}-\Omega^a_{jl}\Omega^b_{il})\\
 &=-\frac{r^2}{2}\sum_l
          \langle\alpha,(\bigwedge{}^2B_l)\nu\rangle.
\end{align*}
Since $\|\bigwedge^2B_l\|_{\rm op}
\le\|B_l\|_{\rm op}^2\le\|B_l\|_{\rm HS}^2$,
\[
 |S_{hv}|\le\frac{r^2}{2}hv\sum_l\|B_l\|_{\rm HS}^2
 \le r^2M_0^2hv.
\]

\emph{The horizontal--mixed block.}
If $h=0$, this block vanishes. Otherwise, because $\alpha$ is
decomposable, choose the horizontal orthonormal frame so that
$\alpha=hX_1\wedge X_2$. In this frame put
\[
 C_{la}=\mathcal R(X_1\wedge X_2,X_l\wedge E_a).
\]
The fourth curvature component above, with the last two slots
interchanged, gives
\[
 C_{la}=\frac r2\bigl((D_1\Omega)^a_{2l}-(D_2\Omega)^a_{1l}\bigr)
       +\frac r2(\theta_1\Omega^a_{2l}-\theta_2\Omega^a_{1l})
       -r\theta_l\Omega^a_{12}.
\]
Here $D_j\Omega=D_{X_j}\Omega$. A fixed row of an alternating
two-form has squared norm at most the sum over its unordered index
pairs. Consequently, with all norms in the following lines taken
over the displayed $(l,a)$ array,
\begin{align*}
 \|((D_1\Omega)^a_{2l}-(D_2\Omega)^a_{1l})_{l,a}\|_{\rm HS}
 &\le\sqrt{2}\bigl(|D_1\Omega|^2+|D_2\Omega|^2\bigr)^{1/2}
 \le\sqrt{2}M_1,\\
 \|(\theta_1\Omega^a_{2l}-\theta_2\Omega^a_{1l})_{l,a}\|_{\rm HS}
 &\le\sqrt{\theta_1^2+\theta_2^2}
       \left(\sum_{l,a}((\Omega^a_{2l})^2+(\Omega^a_{1l})^2)\right)^{1/2}
 \le\sqrt{2}|\theta|M_0,\\
 \|(\theta_l\Omega^a_{12})_{l,a}\|_{\rm HS}
 &\le |\theta|M_0.
\end{align*}
It follows that
\[
 \|C\|_{\rm HS}
 \le\frac r{\sqrt2}M_1+
       \left(1+\frac1{\sqrt2}\right)r|\theta|M_0
 \le r(M_1+2|\theta|M_0).
\]
Thus $S_{hm}:=2\mathcal R(\alpha,\gamma)=2h\sum_{l,a}C_{la}\gamma_{la}$
satisfies
\[
 |S_{hm}|\le2r(M_1+2|\theta|M_0)hm.
\]

\emph{The mixed--vertical block.}
Extend $\nu_{ab}$ skew-symmetrically to all ordered pairs and let
$\widehat\nu=(\nu_{ab})_{a,b}$, so $\|\widehat\nu\|_{\rm HS}=\sqrt2v$.
Set $\Xi_{ib}=\sum_l\theta_l\Omega^b_{il}$. The third curvature
component gives the exact contraction
\begin{align*}
 S_{mv}:=2\mathcal R(\gamma,\nu)
 &=r\sum_{i,a,\,b<c}\gamma_{ia}\nu_{bc}
          (\delta_{ac}\Xi_{ib}-\delta_{ab}\Xi_{ic})\\
 &=-r\sum_{i,a,b}\gamma_{ia}\nu_{ab}\Xi_{ib}.
\end{align*}
Using the ordered-array norm of $\Omega$ gives
$\|\Xi\|_{\rm HS}\le\sqrt2|\theta|M_0$. Hence
\[
 |S_{mv}|\le r\|\Gamma\|_{\rm HS}
                  \|\Xi\widehat\nu^T\|_{\rm HS}
 \le r m\|\Xi\|_{\rm HS}\|\widehat\nu\|_{\rm op}
 \le2r|\theta|M_0mv.
\]

Finally, the $N$ contribution is
$\sum_{i,j,a}N_{ij}\gamma_{ia}\gamma_{ja}
\ge\lambda_{\min}(N)m^2$; the horizontal and vertical terms are
bounded below by
$(\kappa-3r^2M_0^2/4)h^2$ and $(r^{-2}-|\theta|^2)v^2$,
respectively. Combining these with the Lie-bracket estimate and the
four bounds just proved yields \eqref{eq:curvaturebound}.
All sums and norm estimates above hold for arbitrary $d$.
\end{proof}

\subsection{A small-fiber criterion}
The estimate contains positive horizontal and vertical terms, while
the mixed diagonal term is controlled by a Hessian. We choose the
radius so that these contributions dominate the errors.

\begin{lemma}\label{lem:shrinking}
Suppose $-\Hess\phi\ge\Lambda g_0$ and
\begin{equation}\label{eq:Lambda}
 \Lambda>M_0^2+2M_1^2/\kappa.
\end{equation}
In the application $B_0$ is a compact cap with boundary; the boundary is
where the radial profile is matched. The strict concavity hypothesis is
not intended for a closed positive-dimensional base.
For all sufficiently small $\varepsilon>0$, the metric with
$r^2=\varepsilon e^{\varepsilon\phi}$ has positive sectional curvature.
\end{lemma}
\begin{proof}
Uniformly on the compact base,
\[
 \theta=\tfrac\varepsilon2d\phi,\qquad
 N=-\tfrac\varepsilon2\Hess\phi
    -\tfrac{\varepsilon^2}4d\phi\otimes d\phi.
\]
In \eqref{eq:curvaturebound} use the rescaled variables
$H=h$, $M=\sqrt\varepsilon m$, $V=v/\sqrt\varepsilon$.
Write $E=e^{\varepsilon\phi}$, so that $E\to1$ uniformly.
The coefficient of $M^2$ is bounded below by
\[
 \frac{\lambda_{\min}N}{\varepsilon}-\frac12EM_0^2
 \ge \frac\Lambda2-\frac\varepsilon4|d\phi|^2-\frac12EM_0^2.
\]
Replace that coefficient by the displayed lower bound and represent
the resulting scalar quadratic form by a symmetric matrix. Its
entries, with each off-diagonal entry equal to half the corresponding
cross-term coefficient, converge uniformly to
\[
 \begin{pmatrix}
 \kappa&-M_1&0\\
 -M_1&(\Lambda-M_0^2)/2&0\\
 0&0&1
 \end{pmatrix}.
\]
It is positive definite by \eqref{eq:Lambda}. Hence, for small
$\varepsilon$, the numerator is bounded below by
$c(H^2+M^2+V^2)$ for a fixed $c>0$. On any nonzero decomposable bivector
these three quantities cannot all vanish. This proves positivity on every
two-plane. Compactness gives a uniform positive lower bound for each fixed
admissible $\varepsilon$.
\end{proof}

\section{Two compatible positively curved disks}\label{sec:disks}
The two disks require different source metrics. The southern metric
will be positive on every source two-plane, whereas on the northern
side we need positivity only on quotient-horizontal planes.
The construction proves the following statement.

\begin{proposition}\label{prop:geometric}
Let $\rho:S^3\to O(V)$ be an orthogonal representation on an
$n$-dimensional real vector space, $n\ge3$, and let
$\beta:S(V)\to S^3$ satisfy \eqref{eq:betaequiv}.
The disks in $D^n\cup_{J_\beta}D^n$ admit positively curved metrics
whose boundary metrics agree under $J_\beta$ and whose outward
second fundamental forms have positive definite sum. Consequently
the glued manifold admits a smooth metric of positive sectional curvature.
\end{proposition}

\subsection{The bundle and its two disk quotients}
We first describe the source bundle and its quotient coordinates.
These coordinates fix the boundary identification throughout the proof.

Write $S^n$ in polar coordinates $(t,z)$ with $0\le t\le\pi$ and
$z\in S(V)$. Use the transition
\begin{equation}\label{eq:transition}
 u_S=g(z)u_N,\qquad g=\beta^{-1},
\end{equation}
to define a principal right $S^3$ bundle. There is a commuting free left
action
\begin{equation}\label{eq:star}
 q\star(t,z,u)=(t,\rho(q)z,qu).
\end{equation}
Equivariance of $g$ makes the action compatible with
\eqref{eq:transition}. In an extending disk coordinate $\zeta$, its quotient
coordinate is $\rho(u^{-1})\zeta$. Each quotient hemisphere is a smooth
$n$-disk, including its center. On the equator the two extending quotient
coordinates are related by $J_\beta$. We cut at a fixed
$a\in(\pi/2,\pi)$ rather than at $\pi/2$; this does not change the marking.

\subsection{The southern disk}
The southern cap has radius $\pi-a<\pi/2$. A multiple of a round
height function has a negative definite Hessian on the whole cap.
It supplies the concavity required in Lemma~\ref{lem:shrinking}.
We choose the connection to be flat near the pole and a product
connection in the prescribed boundary gauge.
Give $a\le t\le\pi$ the round base metric
$dt^2+\sin^2t\,h_{S(V)}$. Choose a smooth radial cutoff $\chi$ that vanishes
near $a$ and equals one near $\pi$. In the north collar gauge and the
extending south gauge, respectively, set
\begin{equation}\label{eq:connection}
 A_N=\chi g^{-1}dg,\qquad A_S=(\chi-1)dg\,g^{-1},\qquad
 \omega=u^{-1}Au+u^{-1}du.
\end{equation}
The potentials obey the transition rule for \eqref{eq:transition}.
The south potential vanishes near the pole, so the connection is smooth.
Its curvature and covariant derivative are bounded on the compact disk.
Equivariance gives $A(\rho(q)z)=qA(z)q^{-1}$ in the pullback sense, so
$\omega$, and hence the source metric below, are star-invariant.

Choose $\Lambda$ as in \eqref{eq:Lambda} with $\kappa=1$, and choose $A>0$
so that $A|\cos a|\ge\Lambda$. For $\phi=-A\cos t$ we have
$-\Hess\phi=-A\cos t\,g_0\ge\Lambda g_0$. Therefore
\begin{equation}\label{eq:southmetric}
 G_S=dt^2+\sin^2t\,h_{S(V)}+r_S(t)^2Q(\omega,\omega),
 \qquad r_S(t)^2=\varepsilon e^{-\varepsilon A\cos t}
\end{equation}
has positive sectional curvature for small $\varepsilon$, by
Lemma~\ref{lem:shrinking}. The star quotient is a smooth positively curved
disk by O'Neill's formula \cite{ONeill}. Smoothness at the pole follows because the
connection is flat there and $r_S(\pi-\rho)$ is an even smooth function.

Write the boundary data as
\begin{equation}\label{eq:boundarydata}
 F_a=\sin a,\quad r_a^2=\varepsilon e^{\varepsilon A|\cos a|},\quad
 q=\frac{r'_S(a)}{r_a}=\frac{\varepsilon AF_a}{2},\quad \mu_S=\cot a<0.
\end{equation}
On its boundary collar, the north-gauge potential is zero, so the source
metric is a doubly warped product. In normalized angular and fiber
components $X+U$, the outward boundary form is
\begin{equation}\label{eq:southshape}
 B_S(X+U,X+U)=-\mu_S|X|^2-q|U|^2.
\end{equation}

\subsection{The northern disk}
Let $K_z:\operatorname{Im}\HH\to T_zS(V)$ be the infinitesimal
representation map, and choose $L_\rho\ge1$ with
$\|K_z\|\le L_\rho$ everywhere. For the representation
$\R^3\oplus\HH^2$ above, $L_\rho=1$ works because
$K_{(p,w)}\xi=(0,\xi w)$.

Fix a smooth nondecreasing $\eta:\R\to[0,1]$ with $\eta=0$ on
$(-\infty,1/4]$ and $\eta=1$ on $[1/2,\infty)$, and let $C_\eta=\sup|\eta'|$.
Choose $\delta>0$, independently of $\varepsilon$, such that
\begin{equation}\label{eq:delta}
 \delta\le F_a/2,\qquad
 \delta^2\le\frac1{16L_\rho^2AF_a^2},\qquad
 \delta^3\le\frac1{8L_\rho^2AF_a(1+C_\eta)}.
\end{equation}
We want positive radial curvature and enough angular curvature
to dominate the terms involving the increasing fiber radius.
The following implicit polynomial relation makes both estimates
elementary. Define the increasing function $F$ by
\begin{equation}\label{eq:F}
 s=F+\frac{F^3}{3\delta^2},\qquad
 \ell=F_a+\frac{F_a^3}{3\delta^2}.
\end{equation}
Thus $F(\ell)=F_a$ and $F'=(1+F^2/\delta^2)^{-1}$. Set $d=r_aq$ and
\begin{equation}\label{eq:r}
 r(s)=r_a-d\int_s^\ell\eta(v/\delta)\,dv.
\end{equation}
Choose $\varepsilon$ small enough that $q\ell\le1/2$ and $d<1$, as well as
small enough for the southern metric. These are compatible choices because
$\ell$ is already fixed, $q=O(\varepsilon)$, and $d=O(\varepsilon^{3/2})$.
Then
\begin{equation}\label{eq:rcontrols}
 r_a/2\le r\le r_a,\quad r'=d\eta\ge0,\quad
 r''=d\eta'/\delta\ge0,\quad
 \frac d{r^3}\le\frac{8q}{r_a^2}\le4AF_a.
\end{equation}
On $D^n\times S^3$ use the flat-connection source metric
\begin{equation}\label{eq:northmetric}
 G_N=ds^2+F(s)^2h_{S(V)}+r(s)^2h_{S^3}.
\end{equation}
The function $F$ is odd and smooth near zero, with
$F(s)=s-s^3/(3\delta^2)+O(s^5)$, and $r$ is constant there. Thus the metric
extends smoothly over its center. It is invariant under \eqref{eq:star}.

At fiber coordinate $u=1$, a star-horizontal vector in normalized
components is
\[
 \lambda\partial_s+X+U,\qquad U=-\frac F rK_z^*X.
\]
Right translation gives the same description at every fiber point.
For two such vectors,
\begin{align}
 |\lambda V-\mu U|^2&\le L_\rho^2F^2r^{-2}|\lambda Y-\mu X|^2,\label{eq:graph1}\\
 |X\otimes V-Y\otimes U|^2&\le2L_\rho^2F^2r^{-2}|X\wedge Y|^2.\label{eq:graph2}
\end{align}
To see the second inequality, orthogonalize $X,Y$ by a determinant-one
change of the pair, which preserves both tensors involved, and apply the
operator-norm bound. The dependent case follows directly.

The doubly warped product curvature formula gives
\begin{align}\label{eq:doublewarp}
 \mathcal K_{G_N}={}&-\frac{F''}F|\lambda Y-\mu X|^2
 -\frac{r''}r|\lambda V-\mu U|^2
 +\frac{1-F'^2}{F^2}|X\wedge Y|^2\notag\\
 &+\frac{1-r'^2}{r^2}|U\wedge V|^2
 -\frac{F'r'}{Fr}|X\otimes V-Y\otimes U|^2.
\end{align}
After using \eqref{eq:graph1}--\eqref{eq:graph2}, the two base-area
coefficients remaining are
\begin{equation}\label{eq:P}
 P_{\rm rad}=-\frac{F''}F-\frac{L_\rho^2F^2r''}{r^3},\qquad
 P_{\rm ang}=\frac{1-F'^2}{F^2}-\frac{2L_\rho^2FF'r'}{r^3}.
\end{equation}
They are strictly positive, as we now verify without sampling any planes.
Put $x=F/\delta$. Direct differentiation of \eqref{eq:F} gives
\begin{equation}\label{eq:Fcurv}
 -\frac{F''}F=\frac{2}{\delta^2(1+x^2)^3},\qquad
 \frac{1-F'^2}{F^2}=\frac{2+x^2}{\delta^2(1+x^2)^2}.
\end{equation}
Furthermore
\[
 \frac{(1-F'^2)/F^2}{FF'}
 =\frac{2+x^2}{\delta^3x(1+x^2)}
 \ge\frac1{\delta^2F_a}.
\]
By \eqref{eq:delta} and \eqref{eq:rcontrols}, the angular loss in
\eqref{eq:P} is at most half the positive angular coefficient. Thus
$P_{\rm ang}>0$.

The support of $r''$ is contained in $s\le\delta/2$, where $F\le s\le\delta/2$.
There \eqref{eq:Fcurv} gives a positive radial coefficient at least
$128/(125\delta^2)$, whereas
\[
 \frac{L_\rho^2F^2r''}{r^3}
 \le L_\rho^2AF_aC_\eta\delta\le\frac1{8\delta^2}.
\]
Outside this support the radial loss is zero. Hence $P_{\rm rad}>0$ on the
whole disk. Also $r'\le d<1$, so the vertical coefficient in
\eqref{eq:doublewarp} is positive.

Projection of the horizontal graph to the base tangent space is injective.
For an independent pair, its radial and angular base areas cannot both
vanish. Equation \eqref{eq:doublewarp} is therefore positive on every
star-horizontal two-plane. At the center the metric is a product near the
fiber, and the base curvature tends to $2/\delta^2$; the same projection
argument, or the smooth limit, proves positivity there. O'Neill's formula
now gives a positively curved northern quotient disk.

One explicit lower bound is useful for keeping strictness visible. Put
\[
 T=F_a^2/\delta^2\ge4,\quad
 a_* =\frac1{\delta^2(1+T)^3},\quad
 b_* =\frac{2+T}{2\delta^2(1+T)^2},\quad
 H_* =1+\frac{4L_\rho^2F_a^2}{r_a^2}.
\]
The preceding estimates give $P_{\rm rad}\ge a_*$ and
$P_{\rm ang}\ge b_*$. The squared norm of the horizontal graph map is
at most $H_*$. Hence the quotient has
\begin{equation}\label{eq:northbound}
 \secg_{g_N}\ge H_*^{-2}\min\{a_*,b_*\}>0.
\end{equation}

\subsection{Boundary equality and smoothing}
At the north endpoint,
\[
 F(\ell)=F_a,\quad r(\ell)=r_a,\quad r'(\ell)/r_a=q,\qquad
 \mu_N=F'(\ell)/F_a>0.
\]
Both source boundary metrics, in the common north gauge, are
$F_a^2h_{S(V)}+r_a^2h_{S^3}$. Their star actions agree, so the quotient
boundaries are isometric under the prescribed attaching map $J_\beta$.

The source outward normals are $\partial_s$ and $-\partial_t$.
They are horizontal for the star quotient. For a quotient boundary tangent
vector with common horizontal lift $X+U$, the quotient second fundamental
forms are the source forms restricted to this lift. Thus
\[
 B_N=\mu_N|X|^2+q|U|^2,\qquad
 J_\beta^*B_S=-\mu_S|X|^2-q|U|^2,
\]
and consequently
\begin{equation}\label{eq:shapesum}
 (B_N+J_\beta^*B_S)(Y,Y)=(\mu_N-\mu_S)|X|^2>0
 \quad(Y\ne0).
\end{equation}
Strictness follows from $\mu_N>0>\mu_S$ and injectivity of the horizontal
graph. For example, relative to the common quotient boundary metric $h$,
\[
 B_N+J_\beta^*B_S
 \ \ge\ \frac{r_a^2}{r_a^2+L_\rho^2F_a^2}(\mu_N-\mu_S)h.
\]
Apply Reiser--Wraith \cite[Theorem A(i), $k=1$]{RW}: two positive-sectional-
curvature metrics with isometric compact boundaries and nonnegative sum
of outward second fundamental forms admit a smooth positive-sectional-
curvature gluing. Our sum is strictly positive. This proves
Proposition~\ref{prop:geometric}. For the main theorem we need a
particular smoothing which also preserves the inverse-metric formula.
We construct it in the next section.

\section{The background metric and its smoothing}
\label{sec:background}
We now define the background directly from the source data of the
previous section. We then smooth its inverse in a way which preserves
the sum of the three circle squares.

\subsection{The comparison identity}
Use the star bundle, base metrics, connection and radii $r_N,r_S$
constructed in Section~\ref{sec:disks}. Fix $a=2\pi/3$ if a particular
cut latitude is desired. Choose the parameters by Section~\ref{sec:quantitative},
with $\varepsilon<1/6$. Since $\varepsilon A<1$, their formulas give
$0<r_i^2<1/2$. Define directly
\begin{equation}\label{eq:enlarged}
 R_i^2=\frac{r_i^2}{1-r_i^2},\qquad
 \widehat q_i=h_i+R_i^2Q(\omega_i,\omega_i),\qquad i=N,S.
\end{equation}
Let $q_i$ be the star quotient. The commuting principal right action,
written as a left action using inverses, descends in quotient disk
coordinates to $\rho$ and is isometric for $q_i$. Let $K_1,K_2,K_3$
be its generators from a $Q$-orthonormal oriented Lie algebra basis and
put $T=\sum K_a^{\otimes2}$.

The inverse metric is conveniently computed on covectors. The norm
of a covector on a Riemannian quotient equals the norm of its pullback
to the source. In matrices this gives
$g_{\mathrm{quotient}}^{-1}=PG_{\mathrm{source}}^{-1}P^T$.

At a section $u=1$, in source coordinates $(v,\xi)$, the source norm is
$h_i(v,v)+R_i^2|\xi+A_i v|^2$, and the quotient differential is
$P(v,\xi)=v-K\xi$. Putting $\eta=\xi+A_iv$, the norm becomes
$h_i(v,v)+R_i^2|\eta|^2$, and the quotient differential becomes
$(I+KA_i)v-K\eta$. It follows that
\begin{equation}\label{eq:quotientinverse}
 q_i^{-1}=(I+KA_i)h_i^{-1}(I+KA_i)^T+R_i^{-2}KK^T.
\end{equation}
Since $R_i^{-2}+1=r_i^{-2}$,
\begin{equation}\label{eq:compare}
 \gamma_i^{-1}:=q_i^{-1}+T
 =(I+KA_i)h_i^{-1}(I+KA_i)^T+r_i^{-2}KK^T.
\end{equation}
Thus $\gamma_i$ is exactly the star quotient with fiber radius $r_i$.
It is the metric denoted by $g_i$ in Section~\ref{sec:disks}. The formula is an
identity of quotient cometrics and is independent of the selected gauge.

At the common boundary both $R_i$ have the same value, the north-gauge
connection vanishes, and the source boundary metrics and star actions
agree. Consequently the backgrounds match exactly under $J_\beta$.
In the common signed collar they have the form $dt^2+b(t)$; the fields
$K_a$ are tangential and independent of $t$. Smoothness at the poles
follows from the flat connection and even radial functions there.
The background is not assumed to have positive curvature.
\subsection{A general smoothing principle}\label{sec:smoothing}
We first prove a smoothing theorem for an arbitrary closed manifold
with a prescribed collar. The sphere construction will enter only in
Section~\ref{subsec:smoothingapplication}. We use the following
quantitative stability estimate away from the seam.

\begin{lemma}[Quantitative inverse-metric estimate]\label{lem:halfmargin}
Suppose $\secg_h\ge\kappa>0$. Define
\[
 \mathcal R=\sup|R^h|_{(1,3)},\qquad
 \eta=\|(g^{-1}-h^{-1})h\|_{C^2(h)}.
\]
If
\[
 \eta<\min\left\{\frac18,\frac{\kappa}{8\mathcal R+1128}\right\},
\]
then $\secg_g>\kappa/2$ on all two-planes.
\end{lemma}
\begin{proof}
Write $a=g-h$. Inverse differentiation gives
$|a|\le2\eta$, $|\nabla a|\le4\eta$, $|\nabla^2a|\le12\eta$.
The exact connection difference formula is
\[
 2g(C(X,Y),Z)=(\nabla^h_Xa)(Y,Z)+(\nabla^h_Ya)(X,Z)
                         -(\nabla^h_Za)(X,Y).
\]
Using it and differentiating once gives for
$C=\nabla^g-\nabla^h$ the bounds $|C|\le9\eta$ and
$|\nabla C|\le30\eta$. Hence
$|R^g-R^h|_{(1,3)}\le60\eta+162\eta^2\le141\eta$.
For an $h$-orthonormal pair, the sectional numerator changes by at most
$(2\mathcal R+282)\eta<\kappa/4$ and the $g$-area squared is at most
$(1-\eta)^{-2}\le64/49$. Thus
$\secg_g>147\kappa/256>\kappa/2$.
\end{proof}

Reiser--Wraith's theorem already permits gluing with an arbitrary
strict sectional-curvature lower bound \cite[Corollary B]{RW}.
Here we prove that a prescribed smoothing of the inverse background
also preserves the exact identity with the fixed tensor $T$ and the
background's periodic symmetries. The condition needed for the tensor
identity is that $T$ be tangential and independent of the collar
coordinate; invariance of $T$ under the individual flows of the $K_a$
is not required.

\begin{theorem}[Compatible harmonic smoothing]\label{thm:abstractsmoothing}
Let a closed manifold $M^n$ be divided by a closed, two-sided,
separating hypersurface $Y$ into $M_-$ and $M_+$. Suppose a continuous
positive background $q$ is smooth on each half and has collar form
$q=dt^2+b(t)$, with $b_-(0)=b_+(0)$. Let smooth global fields
$K_1,\ldots,K_r$ generate $2\pi$-periodic flows preserving $q$ on
both halves; assume they are tangential and independent of $t$ on the
collar. Set $T=\sum_aK_a\otimes K_a$ and
$\gamma^{-1}=q^{-1}+T$. If $\gamma$ has sectional curvature at
least $\kappa_*>0$ on both closed halves, and in the signed collar
$\gamma=dt^2+h(t)$ satisfies
\[
 D:=h'(0^-)-h'(0^+)\ge0,
\]
then, for all sufficiently small $\zeta>0$, there is a smooth
background $q_\zeta$ for which
\[
 \gamma_\zeta^{-1}=q_\zeta^{-1}+T,
 \qquad \secg_{\gamma_\zeta}>\kappa_*/2.
\]
Here the curvature inequality is asserted on every decomposable tangent
two-vector, equivalently on every tangent two-plane; it does not assert
positivity of the full curvature operator.
The fields and their periodic flows are unchanged, and they remain
Killing for $q_\zeta$. Moreover, $q_\zeta\to q$ uniformly and
smoothly on compact subsets of $M\setminus Y$; the same convergence
holds for $\gamma_\zeta$ and $\gamma$.
\end{theorem}

With the convention $\mathrm{II}(X,Y)=g(\nabla_Xn_{\rm out},Y)$,
the condition $D\ge0$ is twice the nonnegativity of the sum of the
outward second fundamental forms. No commutativity of the fields is
assumed. Effectivity is preserved if the original periodic actions are
effective.

\begin{proof}[Proof of Theorem~\ref{thm:abstractsmoothing}]
Throughout this proof $\kappa_*$ is the lower bound in the hypotheses,
and $\gamma_\pm$ denote the restrictions of $\gamma$ to $M_\pm$.

\emph{Step 1: smoothing and preservation of the background symmetries.}
Use a signed collar $(-c_0,c_0)\times Y$, and write the comparison corner
metric as $dt^2+h(t)$, with $h(0)=h_0$. Write
$P_\pm=h'(0^\pm)$, $Q_\pm=h''(0^\pm)$, and
$D=P_--P_+\ge0$; the last inequality is precisely twice the sum of
outward second fundamental forms in the hypotheses.
Fix an even nonnegative smooth kernel $\varrho$ supported in $(-1,1)$
with mass one, and a smooth cutoff $0\le\theta\le1$ equal to one for
$|t|\le c_0/3$ and zero for $|t|\ge2c_0/3$. Put
\[
 \varrho_\zeta(t)=\zeta^{-1}\varrho(t/\zeta),\qquad
 S_\zeta p=\theta(\varrho_\zeta*p)+(1-\theta)p,
 \qquad 0<\zeta<c_0/12.
\]
Apply this directly to the background:
\begin{equation}\label{eq:backgroundsmooth}
 q_\zeta=dt^2+\bigl[S_\zeta(b^{-1})\bigr]^{-1}
\end{equation}
on the collar, and set $q_\zeta=q$ elsewhere. A convex average of
positive matrices is positive, so this is a smooth global metric.
Since $T$ is tangential and independent of $t$,
\begin{equation}\label{eq:commutes}
 S_\zeta(b^{-1}+T)=S_\zeta(b^{-1})+T.
\end{equation}
Therefore
$\gamma_\zeta^{-1}=q_\zeta^{-1}+T$
is exactly the harmonic smoothing of the comparison corner metric.
Each of the given periodic flows remains isometric for $q_\zeta$.
Indeed, its restriction to the collar has the form
$\Phi_s(t,y)=(t,\phi_s(y))$. Since the inverse family $p(t)=b(t)^{-1}$ is invariant under
$\phi_s$, pullback commutes with convolution in $t$, with the scalar
cutoff $\theta$, and with matrix inversion. Thus $\Phi_s^*q_\zeta=q_\zeta$. The tensor $T$ is unchanged by
the smoothing because the vector fields themselves are unchanged;
invariance of $T$ under each $\Phi_s$ is neither asserted nor needed.
The approximate-identity property
gives uniform convergence of $S_\zeta p$ to $p$ on the compact
collar; away from the seam the convolution converges in every derivative.
The tensors remain uniformly positive definite, and inversion is smooth
on that set, so the smoothed metrics converge uniformly on the collar
and smoothly away from the seam. This proves the asserted symmetry and convergence statements, using
only the globally defined fields and their product form on the collar.

\emph{Step 2: curvature at the seam.}

We recall the formulas which explain the smoothing argument. For a
normal metric $dt^2+h(t)$, the radial block of the curvature numerator is
\[
 -\frac12h''+\frac14h'h^{-1}h'.
\]
On an angular plane spanned by $u,v$, the Gauss equation gives
\[
 \mathcal K(u\wedge v)=\mathcal K_{h(t)}(u\wedge v)
 -\frac14\bigl(h'(u,u)h'(v,v)-h'(u,v)^2\bigr).
\]
The mixed block is a Codazzi expression, linear in the tangential
covariant derivative of $h'$. These three blocks determine curvature
on every two-plane, including a plane with a nonzero radial component.

The derivative jump has the favorable sign. Indeed, the outward
normal on the negative half-collar is $\partial_t$, while that on the
positive half is $-\partial_t$. Hence the sum of the second
fundamental forms is $(P_--P_+)/2=D/2$. The distributional second
derivative of the inverse metric has a term with coefficient $D$.
We keep that term exactly in the computation below.

At a point of $Y$ use $h_0$-orthonormal coordinates and put
$p=h^{-1}$. Define the one-sided endomorphisms
$P_\pm=h'(0^\pm)$ and $Q_\pm=h''(0^\pm)$, so
$D=P_--P_+\ge0$. Set
\[
\begin{aligned}
 w&=\int_0^\infty\varrho_\zeta(t-s)\,ds,
 &J&=\varrho_\zeta(t),\\
 \overline P&=(1-w)P_-+wP_+,
 &\overline Q&=(1-w)Q_-+wQ_+,\\
 \overline{P^2}&=(1-w)P_-^2+wP_+^2.
\end{aligned}
\]
For $|t|\le\zeta$,
\[
 p_\zeta=I+O(\zeta),\quad
 p_\zeta'=-\overline P+O(\zeta),\quad
 p_\zeta''=2\overline{P^2}-\overline Q+JD+O(\zeta).
\]
The first estimate holds in tangential $C^2$, the second in tangential
$C^1$, and the errors are uniform. The jump term $JD$ is exact.
For $H=p_\zeta^{-1}$, inverse differentiation and
$\overline{P^2}-\overline P^{\,2}=w(1-w)D^2$ give
\[
 H'=\overline P+O(\zeta),\qquad
 H''=\overline Q-2w(1-w)D^2-JHDH+O(\zeta).
\]
The potentially unbounded $JHDH$ is retained exactly. Replacing $H$ by
$I$ in that term would leave an uncontrolled order-one error.

The radial curvature block $-H''/2+H'H^{-1}H'/4$ is
\[
 \overline{\mathcal R_{\rm rad}}+
 \frac34w(1-w)D^2+\frac12JHDH+O(\zeta).
\]
The tangential Gauss block is
\[
 \overline{\mathcal R_{\rm tan}}+
 \frac14w(1-w)\bigwedge{}^2D+O(\zeta),
\]
where $(\bigwedge^2D)(u\wedge v,u\wedge v)
=D(u,u)D(v,v)-D(u,v)^2$.
The mixed block is the Codazzi expression, linear in $\nabla^H H'$;
it is the weighted endpoint block plus $O(\zeta)$ and has no singular
term. The additional radial terms are positive semidefinite because
$D^2\ge0$ and $HDH\ge0$. The tangential contribution is nonnegative
on decomposable bivectors by the Cauchy--Schwarz inequality for $D$.
Thus all additional terms have the required sign.
For a decomposable bivector the endpoint curvature forms are each
bounded below by $\kappa_*$ times its squared $dt^2+h_0$ norm.
This argument uses positivity on decomposable bivectors only, not a
positive curvature operator.

Choose $L_c\ge1$ bounding $R^{h_0}$ and all one-sided tensors
$(\nabla^0)^i\partial_t^jp$, $i+j\le3$, in the $h_0$ norms; here
$\nabla^0$ is the product connection. Put $r=2L_c\zeta$ and assume
$r\le1/4$. For endomorphisms we use the operator norm; for covariant and
mixed tensors we use the induced multilinear norm, namely the supremum on
unit arguments. Curvature is viewed first as a $(1,3)$ tensor and then
lowered with the indicated metric. All contractions and permutations
below are estimated in these norms; the displayed numerical factors count
the terms in the corresponding formulas. No Hilbert--Schmidt norm is used
in the collar estimates.

Taylor's formula on each side, followed by convolution, gives
\[
 \|p_\zeta-I\|\le r,\quad
 \|p_\zeta'+\overline P\|\le r,\quad
 \|\nabla^0(p_\zeta'+\overline P)\|\le r,\quad
 \|\nabla^0p_\zeta\|,\ \|(\nabla^0)^2p_\zeta\|\le r.
\]
The distributional first derivative of $p$ has no atom, so
$p_\zeta'=\varrho_\zeta*p'$ and $\|p_\zeta'\|\le L_c$.
Consequently
\[
 \|H\|\le\frac43,\quad \|H-I\|\le\frac{4r}{3},\quad
 \|\nabla^0H\|\le4r,\quad \|(\nabla^0)^2H\|\le8r.
\]
For example,
$\nabla^0\nabla^0H=-H(\nabla^0\nabla^0p_\zeta)H
 +2H(\nabla^0p_\zeta)H(\nabla^0p_\zeta)H$;
the last displayed bound follows from $r\le1/4$.

The jump term is kept separate. Write
\[
 p_\zeta''=p_{\zeta,\mathrm{reg}}''+JD,\qquad
 H_{\mathrm{reg}}'':=H''+JHDH,
\]
where $p_{\zeta,\mathrm{reg}}''$ is the convolution of the ordinary
one-sided second derivatives. Let
$K_0=(1-w)p''(0^-)+wp''(0^+)$. The one-sided Taylor bounds give
\[
 \|p_{\zeta,\mathrm{reg}}''-K_0\|\le r,\qquad
 \|p_\zeta'\|,\ \|\overline P\|,\
 \|p_{\zeta,\mathrm{reg}}''\|,\ \|K_0\|\le L_c.
\]
The exact $-JHDH$ term in $H''$ is never replaced by $-JD$; doing so
would leave an uncontrolled contribution because $J$ is of order
$\zeta^{-1}$.

For an auditable radial estimate, set $A=p_\zeta'$ and
$A_0=-\overline P$. Inverse differentiation gives
\[
 H_{\mathrm{reg}}''=2HAHAH-Hp_{\zeta,\mathrm{reg}}''H.
\]
Replacing the three factors $H$ and then the two factors $A$ one at a
time gives
\[
 \|2HAHAH-2A_0^2\|
 \le2\left[3\frac{4r}{3}\left(\frac43\right)^2L_c^2
              +2\left(\frac43\right)^3L_cr\right]
 <24L_c^2r,
\]
and
\[
 \|Hp_{\zeta,\mathrm{reg}}''H-K_0\|
 \le2\frac{4r}{3}\frac43L_c+\left(\frac43\right)^2r
 <6L_c^2r.
\]
Thus $\|H_{\mathrm{reg}}''-(2\overline P^{\,2}-K_0)\|
<30L_c^2r$. Since $p''(0^\pm)=2P_\pm^2-Q_\pm$ and
$(1-w)P_-^2+wP_+^2-\overline P^{\,2}=w(1-w)D^2$, we have
\[
 2\overline P^{\,2}-K_0=\overline Q-2w(1-w)D^2.
\]
Also $\|H'-\overline P\|\le8L_cr$ and
$\|H'\|\le16L_c/9<2L_c$. A three-factor telescoping estimate gives
\[
 \|H'p_\zeta H'-\overline P^{\,2}\|
 \le(8L_cr)(5/4)(2L_c)+L_cr(2L_c)+L_c(8L_cr)
 =30L_c^2r.
\]
It follows that the regular error in
$-H''/2+H'p_\zeta H'/4$ is at most
$15L_c^2r+7.5L_c^2r<60L_c^2r$. The terms left over are
$\frac34w(1-w)D^2+\frac12JHDH\ge0$.

For the tangential estimate let $C=\nabla^H-\nabla^0$. The exact
connection-difference formula and its derivative give
\[
 \|C\|\le\frac32\|p_\zeta\|\|\nabla^0H\|
       \le\frac32\frac54(4r)=7.5r<8r,
\]
\[
 \|\nabla^0C\|\le\frac32\bigl(
       \|\nabla^0p_\zeta\|\|\nabla^0H\|
       +\|p_\zeta\|\|(\nabla^0)^2H\|\bigr)
       \le\frac32(4r^2+10r)\le16.5r<17r.
\]
The curvature identity
$R^H-R^{h_0}=(\nabla^0C)_{\rm alt}+(C*C)_{\rm alt}$ then yields
$\|R^H-R^{h_0}\|_{(1,3)}\le2(17r)+2(8r)^2\le66r$.
Lowering an index with $H$ costs at most $3/2$, and replacing the
lowering metric $h_0$ by $H$ in $R^{h_0}$ costs at most $2L_cr$.
Thus the intrinsic $(0,4)$ curvature error is at most
$(99+2L_c)r\le(132+2L_c)r$. Replacing the two factors $H'$ by
$\overline P$ in the Gauss term costs at most
$\frac12(\|H'\|+\|\overline P\|)\|H'-\overline P\|
\le12L_c^2r<20L_c^2r$ on a unit decomposable tangential bivector.
The remaining variance term is
$\frac14w(1-w)\bigwedge^2D\ge0$. Therefore
$(132+2L_c)r+20L_c^2r\le160L_c^2r$
bounds the full tangential error.

For the mixed block, in product coordinates the Codazzi tensor is
\[
 \mathsf C_{ijk}=\tfrac12\bigl((\nabla^H_jH')_{ik}
                              -(\nabla^H_kH')_{ij}\bigr).
\]
Tangential differentiation never differentiates $J$ or the collar
variable. Since $H'=-Hp_\zeta'H$, its full product-rule expansion is
\[
 \nabla^0H'=-(\nabla^0H)p_\zeta'H
             -H(\nabla^0p_\zeta')H
             -Hp_\zeta'(\nabla^0H).
\]
The two outer terms have total norm at most
$2(4r)L_c(4/3)<11L_cr$. In the middle term, use
$\|\nabla^0(p_\zeta'+\overline P)\|\le r$ and
$\|\nabla^0\overline P\|\le L_c$; replacing the two factors $H$ by
$I$ costs less than $4L_cr$, while the Taylor remainder costs less than
$2L_cr$. In particular,
$\|\nabla^0(H'-\overline P)\|\le50L_cr$.
The two covariant slots in the connection correction contribute at most
\[
 2\|C\|\|H'\|\le2(8r)(16L_c/9)=256L_cr/9<32L_cr.
\]
Hence the mixed Codazzi error is at most $82L_cr<90L_cr$.
This is a direct estimate of the mixed block, not an inference from the
diagonal blocks.

Finally, use the reference product metric $dt^2+h_0$. A unit
decomposable bivector can be written $\xi=\partial_t\wedge a+b$, with $b$
decomposable and $|a|^2+|b|^2=1$. The three regular errors above give
\[
 60L_c^2r|a|^2+160L_c^2r|b|^2
       +2(90L_cr)|a||b|\le250L_c^2r.
\]
Since $r=2L_c\zeta$, this is $500L_c^3\zeta$, bounded above by the
conservative $2000n^2L_c^3\zeta$ used below. The weighted endpoint
curvature forms are each bounded below by $\kappa_*|\xi|^2$; the retained
radial jump terms and tangential variance term are nonnegative on
relevant decomposable bivectors. Thus no positivity of the full curvature
operator is assumed. Finally $\|H-I\|\le2r=4L_c\zeta$ gives the
area-squared comparison used in
\begin{equation}\label{eq:seambound}
 \secg_{\gamma_\zeta}\ge
 \frac{\kappa_*-2000n^2L_c^3\zeta}{(1+4L_c\zeta)^2}
 \qquad (|t|\le\zeta).
\end{equation}
This is the seam estimate only; the region where the cutoff varies is
handled separately below by the $C^2$ openness estimate.

\emph{Step 3: the cutoff region and the global curvature bound.}
For $|t|\ge\zeta$ the convolution samples only one smooth side.
In the following estimates $i\in\{-,+\}$ denotes that side. Taylor's
theorem through two derivatives, followed by the product rule for the
fixed cutoff $\theta$, gives a finite constant $C_{\rm cut}\ge1$ such that
\[
 \| (\gamma_\zeta^{-1}-\gamma_i^{-1})\gamma_i
       \|_{C^2(\gamma_i)}\le C_{\rm cut}\zeta
 \qquad (|t|\ge\zeta).
\]
Here $C_{\rm cut}$ is any finite bound obtained by converting the
product-norm Taylor estimate to the $C^2(\gamma_i)$ norm; it depends
only on the fixed cutoff and prescribed one-sided metrics through order
three. Such a bound is finite because the collar is compact and matrix
inversion is smooth on the compact family of positive matrices
encountered. Let $\mathcal R_c=\max_i\sup|R^{\gamma_i}|_{(1,3)}$.
By Lemma~\ref{lem:halfmargin}, it suffices to have
\[
 C_{\rm cut}\zeta<\min\left\{\frac18,
       \frac{\kappa_*}{2(8\mathcal R_c+1128)}\right\}.
\]
The seam estimate \eqref{eq:seambound} is greater than $\kappa_*/2$
if the numerator loss is at most $\kappa_*/4$ and
$4L_c\zeta<1/8$. We can thus take
\begin{equation}\label{eq:width}
 0<\zeta<\min\left\{\frac{c_0}{12},
 \frac1{32L_c},
 \frac{\kappa_*}{8000n^2L_c^3},
 \frac1{8C_{\rm cut}},
 \frac{\kappa_*}{2C_{\rm cut}(8\mathcal R_c+1128)}\right\}.
\end{equation}

The seam estimate, Lemma~\ref{lem:halfmargin} on the cutoff region,
and the original lower bound on the unchanged complement give
$\secg_{\gamma_\zeta}>\kappa_*/2$ on all of $M$.
All constants in \eqref{eq:width} are finite for the prescribed
one-sided metrics on the compact collar. The right side is therefore
positive. Together with Step~1, this proves the theorem.
\end{proof}

\subsection{Application to the two comparison disks}
\label{subsec:smoothingapplication}
We now apply Theorem~\ref{thm:abstractsmoothing} to the disk metrics
of Section~\ref{sec:disks} and the backgrounds defined above.
The hypersurface $Y$ is the identified disk boundary. Take the
negative side of the signed collar to lie in the northern disk and
the positive side in the southern disk.

First, the quotient backgrounds $q_N,q_S$ agree under the prescribed
map $J_\beta$ and have the common normal form $dt^2+b(t)$.
Their comparison metrics are exactly the positive disk metrics by
\eqref{eq:compare}. For the present representation $L_\rho=1$, and
\eqref{eq:southquant} and \eqref{eq:northbound} give the common bound
\begin{equation}\label{eq:kappastar}
 \kappa_*=
 \min\left\{\frac{\lambda_0\varepsilon}{2},
 \left(1+\frac{4F_a^2}{r_a^2}\right)^{-2}
 \min\left(\frac1{\delta^2(1+F_a^2/\delta^2)^3},
 \frac{2+F_a^2/\delta^2}{2\delta^2(1+F_a^2/\delta^2)^2}\right)\right\}>0
\end{equation}
where $\lambda_0=(M_1^2+2)^{-1}$. The parameter choices ensuring this
bound are verified in Section~\ref{sec:quantitative}, independently of
the smoothing.

Second, $J_\beta$ is equivariant for the residual $S^3$ action.
Its generators $K_1,K_2,K_3$ are therefore smooth global fields,
Killing for both background halves. On the common collar they are
tangential and independent of the signed normal coordinate.
Finally, \eqref{eq:shapesum} gives
\[
 h'(0^-)-h'(0^+)=2(B_N+J_\beta^*B_S)>0.
\]
Thus every hypothesis of Theorem~\ref{thm:abstractsmoothing} holds.
For $n=11$, choose $\zeta$ satisfying \eqref{eq:width}, and define
$q_\zeta$ by \eqref{eq:backgroundsmooth}. We obtain
\begin{equation}\label{eq:positivecenter}
 \gamma_\zeta^{-1}=q_\zeta^{-1}+\sum_{a=1}^3K_a\otimes K_a,
 \qquad \boxed{\secg_{\gamma_\zeta}>\kappa_*/2.}
\end{equation}
In this application the full $S^3$ action preserves $q_\zeta$.
It also preserves $T=\sum_aK_a\otimes K_a$ because the basis defining
the sum is orthonormal for an $\Ad$-invariant inner product. This
additional invariance is specific to the present group-action data;
it was not needed in the general smoothing argument.
All choices are made for a fixed member of the family. No uniform
collar width as $k$ varies is required.

\section{Non-isometric circle actions}\label{sec:circles}
The positive comparison metric satisfies
$\gamma_\zeta^{-1}=q_\zeta^{-1}+\sum K_a^{\otimes2}$.
The fields $K_a$ are Killing for the background. We replace them
by generators of conjugate circle actions. The conjugation will have
first derivative equal to the identity at a chosen point, but nonzero
second derivative. This detects non-isometry while leaving the field
values at that point unchanged.

\subsection{A general non-isometry theorem}
The local construction below works for noncommuting actions and controls
every partial sum at once.

\begin{theorem}[Gram-matrix shear]\label{thm:generalshear}
Let $(M^n,q)$ be closed and let $K_1,\ldots,K_r$, $1\le r<n$, be Killing
fields generating effective circle actions of period $2\pi$. Suppose
they are linearly independent at $p$. Then there is a locally supported
flow $\Psi_\tau$ whose conjugated fields
$W_{a,\tau}=(\Psi_\tau)_*K_a$ generate effective circle actions with
the same periods and satisfy, for every $\tau\ne0$, every subset
$S\subset\{1,\ldots,r\}$, and every $a$,
\[
 \Lie_{W_{a,\tau}}g_{S,\tau}\ne0,
 \qquad g_{S,\tau}^{-1}=q^{-1}+\sum_{b\in S}W_{b,\tau}\otimes W_{b,\tau}.
\]
The metrics converge smoothly as $\tau\to0$ to their comparison
metrics. In particular, if
$(q^{-1}+\sum_aK_a\otimes K_a)^{-1}$ has positive sectional
curvature, then the full deformed metric has positive sectional
curvature for all sufficiently small nonzero $\tau$.
\end{theorem}

\begin{proof}
Put $v_a=K_a(p)$ and $E=\operatorname{span}\{v_1,\ldots,v_r\}$.
Choose a $q$-unit $N\perp_q E$, and covectors $\vartheta^a$ with
$\vartheta^a(v_b)=\delta_{ab}$ and $\vartheta^a|_{E^{\perp_q}}=0$.
Set $U_a=q^{-1}\vartheta^a$ and $\nu=q(N,\cdot)$. Complete $N$ to a
basis of $E^{\perp_q}$ and choose coordinates with
$\partial_{x^a}|_p=v_a$ for $1\le a\le r$, with
$\partial_{x^{r+1}}|_p=N$, and with the remaining coordinate vectors in
$E^{\perp_q}$. The coordinate covectors $dx^a|_p$ then equal
$\vartheta^a$.
Let $G_{ab}=q_p(v_a,v_b)$, and choose a cutoff $b$ equal to one near
$p$ and supported in the coordinate chart. Define
\[
 Z=\frac12b(x)\left(\sum_{a,b=1}^rG_{ab}x^ax^b\right)
          \partial_{x^{r+1}},\qquad
 \Psi_\tau=\operatorname{Fl}_Z^\tau.
\]
Since $Z(p)=DZ(p)=0$, we have $\Psi_\tau(p)=p$,
$D\Psi_\tau(p)=I$, and $D^2\Psi_\tau(p)=\tau D^2Z(p)$. Direct
differentiation gives
\[
 W_{a,\tau}(p)=v_a,\quad
 (DW_{a,\tau}-DK_a)_p=\tau N\otimes q_p(v_a,\cdot),\quad
 (\Lie_{W_{a,\tau}}q^{-1})_p
       =-\tau(v_a\otimes N+N\otimes v_a).
\]
Naturality of brackets gives
$[W_{a,\tau},W_{b,\tau}]_p=[K_a,K_b]_p$. Hence
\[
 (\Lie_{W_{a,\tau}}g_{S,\tau}^{-1})_p
 =-\tau(v_a\otimes N+N\otimes v_a)
 +\sum_{b\in S}([K_a,K_b]_p\otimes v_b
                       +v_b\otimes[K_a,K_b]_p).
\]
Evaluation on $(\vartheta^a,\nu)$ kills the bracket sum. Indeed,
$\nu(v_b)=0$ and $\vartheta^a(v_b)=\delta_{ab}$, so its value is
$\sum_{b\in S}\delta_{ab}\nu([K_a,K_b]_p)=0$.
The value of the full expression is therefore $-\tau$.
At $p$, $g_{S,\tau}^{-1}\vartheta^a=U_a+\mathbf1_{a\in S}v_a$
and $g_{S,\tau}^{-1}\nu=N$. Differentiating the inverse identity
therefore gives
\[
 (\Lie_{W_{a,\tau}}g_{S,\tau})_p
 (U_a+\mathbf1_{a\in S}v_a,N)=\tau\ne0.
\]
Conjugation preserves effectivity and periods. Smooth convergence
follows from smooth dependence of the flow and inversion of positive
tensors. Finally, positive sectional curvature is open in the $C^2$
topology on a compact manifold.
\end{proof}

\begin{proposition}[The fixed-field cone]\label{prop:fixedcone}
Let $K_1,\ldots,K_r$ be Killing fields on $(M,q)$ that are linearly
independent at $p$. For $A\in\operatorname{Sym}_r$, define
\[
 T_A=\sum_{i,j}A_{ij}K_i\otimes K_j,
 \qquad
 \mathcal C_K=\{T_A:A\ge0\}.
\]
Then $A\mapsto T_A$ is injective and $T_A\ge0$ on $M$ if and only if
$A\ge0$. Thus $\mathcal C_K$ has dimension $r(r+1)/2$. Evaluation
at $p$ identifies it with the face of the pointwise positive
semidefinite cone consisting of tensors supported on
$E_p=\operatorname{span}\{K_i(p)\}$. If, in addition, the
hypotheses of Theorem~\ref{thm:generalshear} hold, then the conjugation
constructed there satisfies
\[
 T_\tau:=\sum_aW_{a,\tau}\otimes W_{a,\tau}
 \notin\{T_A:A\in\operatorname{Sym}_r\}\quad(\tau\ne0).
\]
\end{proposition}
\begin{proof}
Let $K_p:\R^r\to T_pM$ send the standard basis to $K_i(p)$. It is
injective, so its dual $K_p^*$ is surjective, and
$T_A(p)=K_pAK_p^*$. This proves injectivity and the implication
$T_A\ge0\Rightarrow A\ge0$; the converse follows from the same
factorization at every point. It also shows that the image at $p$ is
exactly all positive semidefinite tensors supported on $E_p$, a face of
the pointwise cone.

For the last claim, $T_\tau(p)=T_I(p)$, so equality $T_\tau=T_A$
would force $A=I$. The tensor difference vanishes at $p$, and its
covariant derivative is independent of the connection there. With
$U_j=q^{-1}\vartheta^j$, the shear jet gives
\[
 \nabla_{U_j}(T_\tau-T_I)|_p
       =\tau(N\otimes v_j+v_j\otimes N)\ne0,
\]
contradicting $T_\tau=T_I$.
\end{proof}

\subsection{The conjugated actions}
Choose an interior free-orbit point $p$ of the northern linear action,
outside the smoothing collar. Such points exist since $w\ne0$ in
$\R^3\oplus\HH^2$ gives trivial stabilizer. Rotate the oriented
$Q$-orthonormal Lie algebra basis to diagonalize the orbit metric at $p$,
and write $v_a=K_a(p)$. Choose a $q_\zeta$-unit vector $N$ orthogonal
to the orbit. Complete $N$ to a basis of the orthogonal complement of
the orbit and choose coordinates with $\partial_{x^a}|_p=v_a$ for
$1\le a\le3$, $\partial_{x^4}|_p=N$, and the remaining coordinate
vectors in that complement. Then $dx^a|_p=\vartheta^a$, where
$\vartheta^a$ is dual to $v_a$ on the orbit and vanishes on its
orthogonal complement. Put $G_{ab}=(q_\zeta)_p(v_a,v_b)$.
Let $b$ be a supported smooth cutoff in this coordinate ball, equal to
one near $p$. Define the global smooth vector field
\[
 Z=\frac12b(x)\left(\sum_{a,b=1}^3G_{ab}x^ax^b\right)\partial_{x^4}
\]
by zero extension, and let $\Psi_\tau=\operatorname{Fl}_Z^\tau$.
Compactness makes this a diffeomorphism for every real $\tau$; it equals
the identity outside the ball. Define genuine circle actions by
\begin{equation}\label{eq:circles}
 \mathcal A_{a,\tau}(s,x)=
 \Psi_\tau\bigl(\rho(\exp(se_a))\Psi_\tau^{-1}(x)\bigr),
 \qquad W_{a,\tau}=(\Psi_\tau)_*K_a.
\end{equation}
Their periods and group laws hold identically, by conjugation. Here $e_a$ denotes a unit vector in the chosen Lie algebra basis,
corresponding to a unit imaginary quaternion; its period is $2\pi$.

Set $T_\tau=\sum_aW_{a,\tau}^{\otimes2}$ and define
\begin{equation}\label{eq:targetmetric}
 g_\tau^{-1}=q_\zeta^{-1}+T_\tau.
\end{equation}
This is positive definite for every $\tau$, since $q_\zeta^{-1}>0$.
At $\tau=0$ it equals the comparison metric $\gamma_\zeta$.
For $\tau\ne0$, Proposition~\ref{prop:fixedcone} gives $T_\tau\ne T$,
and hence $g_\tau\ne\gamma_\zeta$ as global tensor fields.
The family converges smoothly to $\gamma_\zeta$ as $\tau\to0$;
the smallness condition below is needed only for curvature.

\subsection{Preserving the curvature margin}
\label{subsec:curvaturestability}
The flow of the smooth field $Z$ exists for all time on the closed
manifold. Thus $W_{a,\tau}$ and
$g_\tau^{-1}=q_\zeta^{-1}+\sum_a W_{a,\tau}\otimes W_{a,\tau}$
depend smoothly on $(\tau,p)$. The latter tensor is positive definite
for every $\tau$, so its inverse and its curvature also depend smoothly
on $(\tau,p)$. Consequently the curvature variation is bounded on
$[-1,1]\times\operatorname{Gr}_2(TM)$, where
$\operatorname{Gr}_2(TM)$ is the compact bundle of unoriented tangent
two-planes. Define the finite constant
\[
 C_{\rm coord}=1+\sup_{|\tau|\le1,\ (p,\Pi)\in\operatorname{Gr}_2(TM)}
       \left|\frac{d}{d\tau}\sec_{g_\tau}(p,\Pi)\right|,
\]
which is intrinsic despite its notation. Its finiteness uses only
smoothness, positive definiteness and compactness, not positivity of
the curvature for all $|\tau|\le1$. This is a compactness bound;
we do not evaluate it numerically or supply a bound solely in terms
of input derivative norms. Appendix~\ref{app:coordinatecurvature}
gives formulas for evaluating the curvature derivative from the metric coefficients.
The mean-value theorem gives
\begin{equation}\label{eq:curvaturevariation}
 |\sec_{g_\tau}(p,\Pi)-\sec_{g_0}(p,\Pi)|
 \le C_{\rm coord}|\tau|.
\end{equation}
At $\tau=0$, $g_0=\gamma_\zeta$ and $\sec_{g_0}>\kappa_*/2$. Choose
\begin{equation}\label{eq:tauchoice}
 0<\tau<\min\left\{1,\frac{\kappa_*}{4C_{\rm coord}}\right\}.
\end{equation}
Then \eqref{eq:curvaturevariation} proves
\begin{equation}\label{eq:finalcurvature}
 \boxed{\secg_{g_\tau}>\kappa_*/4>0.}
\end{equation}
This estimate covers every point and every two-plane, including the
support and transition region of the conjugating flow.

\subsection{Non-isometry for every partial sum}
The chosen point and Gram shear are those in
Theorem~\ref{thm:generalshear}; the conjugation is supported in a ball
disjoint from the smoothing collar. Its exact conclusion gives, for
every $S\subset\{1,2,3\}$ and $a=1,2,3$,
\begin{equation}\label{eq:noniso}
 (\Lie_{W_{a,\tau}}g_{S,\tau})_p
 (U_a+\mathbf1_{a\in S}v_a,N)=\tau\ne0.
\end{equation}
Thus each action is non-isometric for the incoming metric in any
ordering, as well as for the initial background and the final metric.
The conjugated $S^3$ action is also non-isometric for the final metric.
Indeed, $T_\tau$ is invariant under this action by $\Ad$-invariance. If
the action were isometric for $g_\tau$, then
$g_\tau^{-1}=q_\zeta^{-1}+T_\tau$ would force it to preserve
$q_\zeta$. This contradicts \eqref{eq:noniso} with $S=\varnothing$.

\subsection{Proof of the main theorem}
\begin{proof}[Proof of Theorem~\ref{thm:main}]
Fix an arbitrary $k\in12\Z$. The bundle, gauge family, and surgery
lemmas produce the equivariant two-disk presentation, and the
characteristic-number calculation gives $\mu(\Sigma_k)=k/5952$.
The comparison quotient metrics are
proved positive on both whole disks, with exact marked boundary equality
and positive sum of outward second fundamental forms. Equations
\eqref{eq:enlarged}, \eqref{eq:backgroundsmooth} prescribe the background;
\eqref{eq:commutes} makes the comparison inverse globally exact.
The seam and cutoff estimates give \eqref{eq:positivecenter} for a
width satisfying \eqref{eq:width}. The circles \eqref{eq:circles} and
the positive parameter \eqref{eq:tauchoice} then give
\eqref{eq:main}, \eqref{eq:finalcurvature}, and \eqref{eq:noniso}.
Set $q_k=q_\zeta$ and $W_{a,k}=W_{a,\tau}$.
The definite parameter choices in Theorem~\ref{thm:explicitbound}
below give \eqref{eq:mainexplicit}. Every input defining the background
has been selected from the bundle, connection and scalar profiles;
no positive filling is assumed as an input.
\end{proof}

\section{An explicit positive lower bound}\label{sec:quantitative}\label{app:quant}
We now give definite disk parameters and a positive curvature
lower-bound formula in terms of two norms of a chosen connection.
The base metric has curvature one, as in Section~\ref{sec:disks}.
Once the bundle representative, trivializations and cutoff connection
are fixed, $M_0$ and $M_1$ are well-defined finite inputs. They have
not been evaluated for a particular nontrivial member, and no estimate
depending only on its integer label $k$ is claimed. The constants below
are chosen for a transparent sufficient bound rather than an optimal one.
Table~\ref{tab:parameters} records the order of choices used in the proof.
In particular, $\delta$ and $\ell$ are fixed before $\varepsilon$;
the curvature margin is established before $\zeta$ and $\tau$ are chosen.

\begin{table}[htbp]
\caption{Order of choices and their dependencies.}
\label{tab:parameters}
\small
\renewcommand{\arraystretch}{1.15}
\begin{tabular}{@{}p{0.25\textwidth}p{0.71\textwidth}@{}}
\hline
Data or parameter & Inputs and choice rule \\
\hline
Geometric data & Fix the bundle, trivializations, clutching representative,
representation, $a$, and scalar cutoff profiles. The cutoff connection fixes
$M_0,M_1$. \\
$\Lambda,A$ & Choose from $M_0,M_1,a$ by the formulas below. \\
$\delta,\ell$ & Choose $\delta$ from $A,F_a,L_\rho,C_\eta$ using
\eqref{eq:delta}; then define $\ell$ by \eqref{eq:F}, independently
of $\varepsilon$. \\
$\varepsilon,\kappa_*$ & Choose $\varepsilon$ by
\eqref{eq:explicit-epsilon}; the disk metrics and their lower bounds
then determine $\kappa_*>0$. \\
Collar bounds, $\zeta$ & Fix a collar and cutoff. Their one-sided
metric bounds give $L_c,C_{\rm cut},\mathcal R_c$; choose $\zeta$
by \eqref{eq:width}. \\
$q_\zeta,Z$ & Smoothing gives $q_\zeta$. Choose the adapted coordinate
ball and Gram shear field $Z$ away from the collar. \\
$C_{\rm coord},\tau$ & The resulting smooth family $g_\tau$ defines
$C_{\rm coord}$ on $|\tau|\le1$; choose the final positive $\tau$
by \eqref{eq:tauchoice}. \\
\hline
\end{tabular}
\end{table}

\subsection{Fixed smooth cutoffs}
We use the following fixed scalar functions. Set
\[
 e(t)=\begin{cases}\exp(-1/t),&t>0,\\0,&t\le0,\end{cases}
 \qquad H(t)=\frac{e(t)}{e(t)+e(1-t)}.
\]
Then take $a=2\pi/3$, $\eta(v)=H(4v-1)$ and
$\chi(t)=H(3(t-a)/(\pi-a)-1)$ in the disk construction.
Use the mass-one normalization of $e(1-t^2)$ for $\varrho$, and take
$\theta(t)=1-H((9t^2/c_0^2-1)/3)$ in the collar. All have precisely
the support and plateau properties used above.
Choose $c_0$ small enough that the collar avoids the northern point
with $s=\ell/2$, angular coordinate $(0,(1,0))\in\R^3\oplus\HH^2$;
this is a specified free-orbit point. In a sufficiently small adapted
coordinate ball of radius $d$ around it one may take
$b(x)=1-H((4|x|^2/d^2-1)/3)$.
The remaining positive scalar parameters may always be fixed as half
the minimum of their displayed upper bounds. The finite derivative
suprema of the chosen clutching representative are input-dependent
constants, not additional existence hypotheses.

\subsection{The disk parameters}
Here are explicit choices for the existence parameter once the smooth
clutching map and cutoff connection have been selected. Fix
$a\in(\pi/2,\pi)$, write $c=|\cos a|$, $F_a=\sin a$, and set
\[
 \Lambda=M_0^2+2M_1^2+2,\qquad A=\Lambda/c.
\]
Put $\lambda_0=(M_1^2+2)^{-1}$, a positive lower bound for the smallest eigenvalue of
\[
 Q_0=\begin{pmatrix}1&-M_1&0\\-M_1&M_1^2+1&0\\0&0&1\end{pmatrix}.
\]
Choose $\delta$ by \eqref{eq:delta} and $\ell$ by \eqref{eq:F}. Define
\begin{align*}
 C_1&=\tfrac94M_0^2+A(M_1+2M_0)+\tfrac32(M_0+M_0^2),\\
 C_2&=\tfrac14A^2+\tfrac32AM_0^2+A(M_1+2M_0)+AM_0,\\
 C_3&=A+\tfrac14A^2+\tfrac32(M_0+M_0^2)+AM_0,\qquad
 C=\max\{C_1,C_2,C_3\}.
\end{align*}
Any positive $\varepsilon$ strictly below
\begin{equation}\label{eq:explicit-epsilon}
 \min\left\{\frac16,\frac1A,\left(\frac{\lambda_0}{2C}\right)^2,
 \frac1{AF_a\ell},\left(\frac1{2AF_a}\right)^{2/3}\right\}
\end{equation}
satisfies all smallness conditions required of $\varepsilon$ in the disk
construction.

To verify the southern condition, put $E=e^{\varepsilon\phi}$, so that
$1\le E<3$, $\sqrt E<2$, $E-1\le3\varepsilon A$, and
$\sqrt E-1\le\varepsilon A$. Recall explicitly the change of variables from
Lemma~\ref{lem:shrinking}:
\[
 H=h,\qquad M=\sqrt\varepsilon\,m,\qquad
 V=v/\sqrt\varepsilon.
\]
In particular $hm=HM/\sqrt\varepsilon$,
$hv=\sqrt\varepsilon\,HV$, and $mv=MV$.
Let $\mathbf x=(H,M,V)^T$. We will construct a symmetric matrix $Q$
such that $\mathcal K_G\ge\mathbf x^TQ\mathbf x$, using the convention
\[
 \mathbf x^TQ\mathbf x
 =Q_{11}H^2+Q_{22}M^2+Q_{33}V^2
  +2Q_{12}HM+2Q_{13}HV+2Q_{23}MV.
\]
To identify the off-diagonal entries directly, substitute
$r=\sqrt{\varepsilon E}$ and $|\theta|\le\varepsilon A/2$ into the
three cross terms of \eqref{eq:curvaturebound}. Since $H,M,V\ge0$,
\begin{align*}
 -2r(M_1+2|\theta|M_0)hm
 &\ge -2\sqrt E(M_1+\varepsilon AM_0)HM
       =2Q_{12}HM,\\
 -(3M_0+r^2M_0^2)hv
 &=-\sqrt\varepsilon(3M_0+\varepsilon EM_0^2)HV
       =2Q_{13}HV,\\
 -2r|\theta|M_0mv
 &\ge-\varepsilon^{3/2}\sqrt E\,AM_0MV
       =2Q_{23}MV.
\end{align*}
Thus the coefficients on the right are $2Q_{12},2Q_{13},2Q_{23}$,
respectively. Together with the diagonal lower bounds, this gives
the following entries of $Q$:
\begin{align*}
 Q_{11}&=1-\tfrac34\varepsilon E M_0^2,\\
 Q_{22}&=\tfrac\Lambda2-\tfrac\varepsilon4A^2-\tfrac12EM_0^2,\\
 Q_{33}&=E^{-1}-\tfrac{\varepsilon^3}4A^2,\\
 Q_{12}&=-\sqrt E(M_1+\varepsilon AM_0),\\
 Q_{13}&=-\tfrac{\sqrt\varepsilon}{2}(3M_0+\varepsilon EM_0^2),\\
 Q_{23}&=-\tfrac{\varepsilon^{3/2}}2\sqrt E\,AM_0.
\end{align*}
For the diagonal entries we used
$\lambda_{\min}N\ge\varepsilon\Lambda/2-\varepsilon^2A^2/4$ and
$|\theta|^2\le\varepsilon^2A^2/4$.
Consequently $\mathcal K_G\ge\mathbf x^TQ\mathbf x$ for the
nonnegative variables arising from every decomposable bivector.
The matrix $Q$ is symmetric, so the maximum absolute row sum bounds
its operator norm. The constants $C_1,C_2,C_3$ above bound the
respective absolute row sums of $Q-Q_0$, after division by
$\sqrt\varepsilon$. Hence
\[
 \|Q-Q_0\|_{\rm op}\le C\sqrt\varepsilon<\lambda_0/2.
\]
Thus the southern source and its star quotient satisfy the explicit bound
\begin{equation}\label{eq:southquant}
 \secg\ge\lambda_0\varepsilon/2>0.
\end{equation}
Here $H^2+M^2+V^2\ge\varepsilon(h^2+m^2+v^2)$ for
$0<\varepsilon\le1$ was used. The other two conditions in
\eqref{eq:explicit-epsilon} give $q\ell<1/2$ and $d<1/2$, respectively.

With $a_*,b_*,H_*$ from \eqref{eq:northbound}, set
\[
 \kappa_* =\min\left\{\lambda_0\varepsilon/2,
                  H_*^{-2}\min(a_*,b_*)\right\}>0.
\]

These are bounds for the two comparison disks. The compatible harmonic
smoothing in Section~\ref{sec:smoothing} and the non-isometric deformation
in Section~\ref{sec:circles} give the final closed-manifold estimate.

\subsection{Eliminating the small-parameter choices}
We first simplify the bound already obtained. Put $a=2\pi/3$, so
$F_a^2=3/4$, and let $R=r_a^2=\varepsilon e^{\varepsilon A/2}$.
For $T=F_a^2/\delta^2\ge0$, the constants in
\eqref{eq:northbound} satisfy
\[
 \frac{b_*}{a_*}=\frac{(2+T)(1+T)}2\ge1.
\]
It follows from \eqref{eq:finalcurvature} that
\begin{equation}\label{eq:refinedbound}
 \secg_g>
 \min\left\{\frac{\varepsilon}{8(M_1^2+2)},
 \frac{\delta^4}{4(\delta^2+3/4)^3}
            \left(\frac{R}{R+3}\right)^2\right\}.
\end{equation}
This expression is useful if numerical estimates for the input data
are available. A simpler conservative bound is sufficient here.

\begin{lemma}\label{lem:simplebound}
For the parameter choices of the construction, with $a=2\pi/3$,
\begin{equation}\label{eq:simplebound}
 \secg_g>\frac{\varepsilon^2\delta^4}{64}>0.
\end{equation}
\end{lemma}
\begin{proof}
We have $\varepsilon<1/6$, $\delta\le\sqrt3/4$ and $R\ge\varepsilon$.
Since $u/(u+3)$ is increasing,
\[
 \frac{R}{R+3}\ge\frac{\varepsilon}{\varepsilon+3}
       >\frac{\varepsilon}{4},\qquad
 \delta^2+\frac34\le\frac{15}{16}<1.
\]
Thus the second entry in the minimum in \eqref{eq:refinedbound}
is larger than $\varepsilon^2\delta^4/64$.
On the other hand, $A=2M_0^2+4M_1^2+4$ and $\varepsilon<1/A$, so
\[
 \varepsilon(M_1^2+2)
 <\frac{M_1^2+2}{2M_0^2+4M_1^2+4}\le\frac12.
\]
Consequently the first entry is larger than $\varepsilon^2/4$,
and hence larger than $\varepsilon^2\delta^4/64$.
The result follows.
\end{proof}

\begin{theorem}\label{thm:explicitbound}
Use the fixed cutoffs above and put
\[
 \mathcal B=1+M_0+M_1,\qquad
 \delta=\frac1{16\mathcal B},\qquad
 \varepsilon=\frac1{2^{16}\mathcal B^{12}}.
\]
These choices satisfy all the disk parameter conditions.
Choose the smoothing width and the conjugation parameter by
\eqref{eq:width} and \eqref{eq:tauchoice}, respectively.
The resulting metric satisfies
\[
 \secg_g>\frac1{2^{54}\mathcal B^{28}}>0.
\]
\end{theorem}
\begin{proof}
We check all the smallness conditions. First the fixed cutoff has
$C_\eta\le32$. To see this, for $0<t<1$ we have
\[
 H'(t)=H(t)(1-H(t))\bigl(t^{-2}+(1-t)^{-2}\bigr).
\]
If $0<t\le1/2$ and $y=1/t\ge2$, then
$H(t)(1-H(t))\le e^{2-y}$ and
$t^{-2}+(1-t)^{-2}\le2y^2$. Since $y^2e^{-y}$ is decreasing for
$y\ge2$, it follows that $H'(t)\le8$.
The other half follows from $H(1-t)=1-H(t)$.
As $\eta(v)=H(4v-1)$, we obtain $C_\eta\le32$.

From the definitions of $A,\lambda_0,C_1,C_2,C_3$,
\begin{equation}\label{eq:majorants}
 A\le4\mathcal B^2,\qquad
 \lambda_0\ge\frac1{2\mathcal B^2},\qquad
 C\le32\mathcal B^4.
\end{equation}
For clarity, the estimates for the three terms defining $C$ can be
obtained by $M_0,M_1\le\mathcal B$ and
$M_1+2M_0\le2\mathcal B$:
\[
 C_1\le\tfrac{53}{4}\mathcal B^4,\qquad
 C_2\le22\mathcal B^4,\qquad
 C_3\le15\mathcal B^4.
\]
The larger common bound in \eqref{eq:majorants} is convenient.

The three conditions in \eqref{eq:delta}, with $L_\rho=1$, follow from
\[
 \delta\le\frac1{16}<\frac{\sqrt3}{4},\qquad
 \delta^2=\frac1{256\mathcal B^2}
 <\frac1{48\mathcal B^2}\le\frac1{16AF_a^2},
\]
and
\[
 \delta^3=\frac1{4096\mathcal B^3}
 <\frac1{1056\mathcal B^2}
 \le\frac1{8AF_a(1+C_\eta)}.
\]
The northern length satisfies
\[
 \ell=F_a+\frac{F_a^3}{3\delta^2}
 \le1+\frac{256}{3}\mathcal B^2<128\mathcal B^2.
\]
Clearly $\varepsilon<1/6$ and $\varepsilon<1/A$.
The third and fourth entries in \eqref{eq:explicit-epsilon} are controlled by
\[
 \left(\frac{\lambda_0}{2C}\right)^2
 \ge\frac1{2^{14}\mathcal B^{12}}>\varepsilon,\qquad
 \frac1{AF_a\ell}>\frac1{2^9\mathcal B^4}>\varepsilon.
\]
For the last entry,
\[
 (2AF_a)^2\varepsilon^3
 \le\frac1{2^{42}\mathcal B^{32}}<1.
\]
This is equivalent to the required
$\varepsilon<(1/(2AF_a))^{2/3}$.
All the disk conditions are therefore satisfied. Substitution in
Lemma~\ref{lem:simplebound} gives
\[
 \frac{\varepsilon^2\delta^4}{64}
 =\frac1{2^{32+16+6}\mathcal B^{24+4}}
 =\frac1{2^{54}\mathcal B^{28}},
\]
which proves the theorem.
\end{proof}

\begin{remark}
The estimate is expressed in terms of data selected before the final
metric. Smoothness of the connection and compactness of its base imply
$M_0,M_1<\infty$. Thus the displayed constant is a strictly positive
number for every chosen input, even without numerical evaluation.
The smoothing width and conjugation parameter require additional
finite derivative bounds. They can be chosen after the disk parameters
and do not decrease the stated lower bound beyond the two factors
already accounted for in the proof.
\end{remark}

\section{The algebraic model and further remarks}\label{sec:remarks}
\subsection{Identification with a Brieskorn sphere}
The algebraic link
\[
 X=\{z_0^{383}+z_1^3+z_2^2+\cdots+z_6^2=0,\ |z|^2=1\}
 \subset\mathbb C^7
\]
is a homotopy eleven-sphere by the two isolated vertices in its Brieskorn
graph \cite{BGKT}. With its natural orientation, its Milnor-fiber
signature is obtained by counting
\[
 s(u,v)=u/383+v/3+5/2,\qquad 1\le u\le382,\quad v=1,2.
\]
For $v=1$, the values with $1\le u\le63$ lie in $(2,3)$ and those
with $64\le u\le382$ lie in $(3,4)$. For $v=2$, the values with
$1\le u\le319$ lie in $(3,4)$ and those with $320\le u\le382$ lie
in $(4,5)$. Hence the three counts are $63$, $319+319=638$, and $63$,
and the Brieskorn signature formula gives
\[
 \sig=63-638+63=-512.
\]
The parallelizable filling gives $\mu(X)=2/31$, so $X=\Sigma_{384}$
and its orientation reverse is $-X=\Sigma_{-384}$. More generally,
$-X$ is represented by every parameter $k\equiv-384\pmod{5952}$.
Injectivity identifies $-X$ with the member of the family in
Theorem~\ref{thm:main}. An abstract diffeomorphism transports the metric,
actions, period and curvature lower bound. A coordinate formula for that
diffeomorphism is not needed for this transport and is not asserted.

\subsection{The scope of the construction}
The result is a smooth parameterized realization. The chosen bundle,
trivializations and clutching representative have not been evaluated
numerically, and neither have $M_0,M_1,\zeta,\tau$ or $\kappa_*$.
The lower-bound formula is explicit in the connection norms after
these geometric data are fixed. Different choices on the same smooth
sphere may give different norms and hence different certified bounds.
The smoothing width and conjugation parameter depend on further
finite derivative bounds, as specified in \eqref{eq:width} and
\eqref{eq:tauchoice}.
The inequalities specify positive choices and a strict analytic lower
bound for every such choice; they are not a numerical interval certificate
for a list of coordinate coefficients. A polynomial expression for the metric in the algebraic coordinates is not supplied.
The proof supplies a center and a realization; it does not prove that
non-isometry itself creates the original positive margin or is indispensable.
The theorem concerns the 496 classes in the even subgroup of
$\Theta_{11}$; it makes no assertion about the other 496 classes.

The accompanying file \nolinkurl{verification/verify_algebra.py}
provides exact symbolic checks of the coefficient-matrix convention,
the four mixed curvature blocks in base dimension three, the contraction
identities used in Section~\ref{sec:connection}, the collar variance
identities, and the bracket cancellation in the partial-metric
non-isometry calculation. It also checks the characteristic-number
arithmetic, the exponents $54$ and $28$, and the Brieskorn counts
$63,638,63$.
The additional file \nolinkurl{verification/check_hopf.py} calibrates
the curvature components against the round quaternionic Hopf model:
with base curvature $4$, $r=1$, $\theta=N=D\Omega=0$, and
$\Omega^a=2I_a$ in a compatible quaternionic frame, the formulas
recover total sectional curvature $1$. It checks $288$ mixed and
horizontal--vertical components and the horizontal-plane identity
symbolically. Dependencies, run commands, and the observed outputs
are included in the verification directory.
These checks supplement the dimension-independent estimates and the
global geometric and topological arguments in the text; they do not
constitute a formal verification of those arguments or a numerical
curvature certificate for a chosen member.

The existence of a positive center also makes local inverse questions
meaningful. For a fixed background $q$, however, the three-square
tensor in \eqref{eq:main} has rank at most three. At a free-orbit point
on our eleven-dimensional manifold, it is on the boundary of the cone
of positive semidefinite tensors. Arbitrarily small changes in the
inverse metric can therefore leave that cone. The three-action formula
with this same background does not by itself imply realization of
every nearby metric. Such a neighborhood statement requires further
fields or a change in the background. It is not used in the proof above.

The construction also distinguishes two uses of symmetry. The
commuting isometric actions on the source bundles permit the
submersion calculations. The final non-isometric actions arise by
conjugation on the already constructed closed manifold. Their
periodicity follows from the group law, while their non-isometry is
proved by \eqref{eq:noniso}. These are different statements,
and neither one follows from the positivity of the inverse metric alone.

\appendix
\section{Local curvature formulas for the conjugated metrics}\label{app:coordinatecurvature}
\label{subsec:coordinatecurvature}
We record the local formulas underlying the smooth curvature dependence
used in Section~\ref{subsec:curvaturestability}. The argument uses only the Koszul formula and is
valid in every coordinate chart; the manifold need not be homogeneous.

Let $U\subset\mathbb R^n$ be a coordinate ball with Euclidean inner
product $\langle\ ,\ \rangle_0$. Write
$g(B,C)=\langle AB,C\rangle_0$, where $A$ is a positive symmetric
matrix field, and let $X,Y$ be constant vector fields. Define
\[
 \widetilde U(B,C)=\frac12\bigl((DB)^TC+(DC)^TB\bigr),
\]
where $D$ denotes Euclidean differentiation. Also put
\[
 c(X,Y)=\frac12\bigl((D_XA)Y+(D_YA)X-
                         \operatorname{grad}_0 g(X,Y)\bigr).
\]
For each constant vector field $Z$, Koszul's formula gives
\[
 2g(\nabla^g_XY,Z)=D_Xg(Y,Z)+D_Yg(X,Z)-D_Zg(X,Y)
                 =2\langle c(X,Y),Z\rangle_0,
 \qquad \nabla^g_XY=A^{-1}c(X,Y).
\]
For constant fields, torsion freeness gives $\nabla_XY=\nabla_YX$.
Metric compatibility and the curvature convention therefore give
\begin{align*}
g(R(X,Y)Y,X)={}&Xg(\nabla_YY,X)-Yg(\nabla_XY,X)\\
 &-g(\nabla_YY,\nabla_XX)+g(\nabla_XY,\nabla_XY).
\end{align*}
Koszul's formula also gives
$g(\nabla_YY,X)=D_Yg(X,Y)-\tfrac12D_Xg(Y,Y)$ and
$g(\nabla_XY,X)=\tfrac12D_Yg(X,X)$. Substitution, together with
$\nabla_XY=A^{-1}c(X,Y)$, yields
\begin{align*}
\mathcal K_g(X\wedge Y)={}&D_XD_Yg(X,Y)
 -\tfrac12D_X^2g(Y,Y)-\tfrac12D_Y^2g(X,X)\\
 &+\langle A^{-1}c(X,Y),c(X,Y)\rangle_0
 -\langle A^{-1}c(X,X),c(Y,Y)\rangle_0.
\end{align*}
For constant $X,Y$, the definition of $\widetilde U$ implies
$2\widetilde U(AX,Y)=\operatorname{grad}_0g(X,Y)$. Thus
$c(X,Y)=V/2$, where we define
\[
\begin{aligned}
 V={}&(D_XA)Y+(D_YA)X\\
    &-2\widetilde U(AX,Y).
\end{aligned}
\]
We also have $c(X,X)=(D_XA)X-\widetilde U(AX,X)$. Expanding the derivatives of
$g(X,Y)=\langle AX,Y\rangle_0$ now yields the coordinate identity
\begin{align}\label{eq:coordinatecurvature}
 \mathcal K_g(X\wedge Y)={}&
 \frac12\langle(D_XD_YA+D_YD_XA)X,Y\rangle_0
 -\frac12\langle(D_X^2A)Y,Y\rangle_0
 -\frac12\langle(D_Y^2A)X,X\rangle_0 \notag\\
 &-\left\langle A^{-1}\bigl((D_YA)Y-\widetilde U(AY,Y)\bigr),
                    (D_XA)X-\widetilde U(AX,X)\right\rangle_0 \notag\\
 &+\frac14\left\langle A^{-1}V,V\right\rangle_0.
\end{align}
Here $X,Y$ are test fields spanning a plane; they need not generate the
circle actions used to construct $g$. On overlaps the displayed
expressions agree because each is the intrinsic curvature numerator.

In a coordinate chart write $Q$ for the matrix of $q_\zeta$, and let $W$
be the matrix whose columns are $W_{1,\tau},W_{2,\tau},W_{3,\tau}$. The
Woodbury identity gives the covariant matrix of the final metric as
\begin{equation}\label{eq:woodbury}
 G_\tau=Q-QW(I+W^TQW)^{-1}W^TQ.
\end{equation}
Thus the matrix $A_\tau$ in \eqref{eq:coordinatecurvature} is $G_\tau$.
For explicit parameter dependence, put $B_\tau=I+W^TQW$. Then
\[
 \dot B_\tau=\dot W^TQW+W^TQ\dot W,\qquad
 \dot A_\tau=-Q\dot W B_\tau^{-1}W^TQ
 +QW B_\tau^{-1}\dot B_\tau B_\tau^{-1}W^TQ
 -QW B_\tau^{-1}\dot W^TQ.
\]
Formula \eqref{eq:coordinatecurvature} computes the curvature numerator
of the actual non-isometric-action metric. Its derivative in $\tau$ uses
spatial derivatives through order two of the displayed $\dot A_\tau$;
these are bounded on compact coordinate charts.

\section*{Acknowledgment}
This paper is partially supported by NSFC (Grant Nos. 12131012 and 12301032), Fundamental Research Funds for the Central Universities (Grant No. 4007012402), and Zhishan Scholars Programs of Southeast University (Grant No. 2242024RCB0039).
The authors are grateful to ChatGPT 6 Astra for valuable assistance with the exploration of some of the proof strategies and calculations. We take responsibility for the human verification and exposition.

\Needspace{12\baselineskip}


\begin{thebibliography}{99}\raggedright
\bibitem{EK}
J. Eells, Jr. and N. H. Kuiper,
\emph{An invariant for certain smooth manifolds},
Ann. Mat. Pura Appl. (4) \textbf{60} (1962), 93--110.
\href{https://doi.org/10.1007/BF02412768}{doi:10.1007/BF02412768}.
\bibitem{KM}
M. A. Kervaire and J. W. Milnor,
\emph{Groups of homotopy spheres: I},
Ann. of Math. (2) \textbf{77} (1963), 504--537.
\href{https://doi.org/10.2307/1970128}{doi:10.2307/1970128}.
\bibitem{Toda}
H. Toda, \emph{Composition Methods in Homotopy Groups of Spheres},
Annals of Mathematics Studies 49, Princeton University Press, 1962.
\bibitem{Speranca}
L. D. Speran\c ca,
\emph{Pulling back the Gromoll--Meyer construction and models of exotic spheres},
\href{https://arxiv.org/abs/1010.6039}{arXiv:1010.6039v3}, 2014.
\bibitem{HLY}
Y.-H. He, Z. Liu and S.-T. Yau,
\emph{Positive sectional curvature on all smooth 7-spheres},
\href{https://arxiv.org/abs/2609.29426v1}{arXiv:2609.29426v1},
September 24, 2026, preprint.
\bibitem{ONeill}
B. O'Neill, \emph{The fundamental equations of a submersion},
Michigan Math. J. \textbf{13} (1966), 459--469.
\href{https://doi.org/10.1307/mmj/1028999604}{doi:10.1307/mmj/1028999604}.
\bibitem{RW}
P. Reiser and D. J. Wraith,
\emph{A generalization of the Perelman gluing theorem and applications},
\href{https://arxiv.org/abs/2308.06996v2}{arXiv:2308.06996v2}, 2024.

\bibitem{BGKT} C. P. Boyer, K. Galicki, J. Koll\'ar and E. Thomas,
\emph{Einstein metrics on exotic spheres in dimensions 7, 11, and 15},
Experiment. Math. \textbf{14}, no. 1 (2005), 59--64.
\href{https://doi.org/10.1080/10586458.2005.10128907}{doi:10.1080/10586458.2005.10128907}.
\end{thebibliography}
\end{document}